\documentclass[11pt]{amsart}
\usepackage{amsmath,amssymb,enumerate,etoolbox,mathrsfs,xcolor,url,hyperref}
\usepackage[normalem]{ulem}

\usepackage[top=30mm,right=25mm,bottom=30mm,left=25mm]{geometry}

\usepackage{amsthm}
\usepackage{tikz-cd}
\usepackage{amsrefs}
\usepackage{hyperref}
\usepackage{mathtools}
\hypersetup{colorlinks=true, linkcolor=black, filecolor=black, urlcolor=black,}

\theoremstyle{plain}

\newtheorem{teo}{Theorem}[section]
\newtheorem{lem}[teo]{Lemma}

\newtheorem{dfn}[teo]{Definition}

\newtheorem*{thmA}{Theorem A}
\newtheorem*{thmB}{Theorem B}
\newtheorem*{thmC}{Theorem C}

\newtheorem*{conj}{Conjecture}

\newcommand{\Syl}{\operatorname{Syl}}

\newcommand{\Gal}{\operatorname{Gal}}

\newcommand{\Aut}{\operatorname{Aut}}
\newcommand{\Out}{\operatorname{Out}}
\newcommand{\Ker}{\operatorname{Ker}}

\newcommand{\Sgn}{\operatorname{Sgn}}

\newcommand{\Irr}{\operatorname{Irr}}
\newcommand{\Inn}{\operatorname{Inn}}

\newcommand{\Lin}{\operatorname{Lin}}

\newcommand{\SL}{\operatorname{SL}}
\newcommand{\SU}{\operatorname{SU}}
\newcommand{\Sp}{\operatorname{Sp}}
\newcommand{\GU}{\operatorname{GU}}
\newcommand{\PSL}{\operatorname{PSL}}
\newcommand{\PSU}{\operatorname{PSU}}
\newcommand{\PGU}{\operatorname{PGU}}
\newcommand{\GL}{\operatorname{GL}}
\newcommand{\PGL}{\operatorname{PGL}}
\newcommand{\Sz}{\operatorname{Sz}}

\begin{document}

\title[Inductive Feit and Galois-McKay conditions for small-rank]
{Inductive Feit and Galois-McKay conditions for some small-rank simple groups of Lie Type} 

\author{Carlos Tapp}
\email{ct675@math.rutgers.edu} 
\address{Department of Mathematics\\
    Rutgers University \\ 
    Piscataway, NJ 08854\\ U.S.A.} 

\begin{abstract} We complete the proof of the inductive Feit condition and the inductive Galois-McKay condition for the simple groups $\PSL_2(q)$. We also prove that the Suzuki groups $^{2}B_2(2^{2n+1})$ satisfy the inductive Feit condition. \end{abstract}
\thanks{I want to express my gratitude to R. Boltje, N. Monteiro and G. Navarro for many useful conversations. Moreover, I would like to thank my advisor P. H. Tiep for suggesting me the problem, giving an excellent guidance and for his unwavering support in every aspect.}
\thanks{This work was done under the support of the NSF grant DMS-2200850, and the Mathematics Department of Rutgers University. }
\thanks{Finally, I want to thank the authors of \cite{BKNT} for providing an early preprint.}

\maketitle 
\section{Introduction}
Recall that the conductor of an abelian number field $K$ is the smallest positive integer $n$ such that $K\subseteq \mathbb{Q}_n$, where $\mathbb{Q}_n$ is the $n$-th cyclotomic field. Let $G$ be a finite group, denote by $\Irr(G)$ the set of complex irreducible characters
and let $\chi\in\Irr(G)$. We denote by $\mathbb{Q}(\chi)$ the smallest field containing the character values. Since character values are sums of roots of unity, $\mathbb{Q}(\chi)$ is an abelian number field. By definition, the conductor of $\chi$ is the conductor of the field $\mathbb{Q}(\chi)$, and we denote it by $c(\chi)$. In 1980 W. Feit proposed the following conjecture.

\begin{conj}[Feit] Let $G$ be a finite group and let $\chi\in\Irr(G)$. Then there exists $g\in G$ such that the order of $g$, $o(g)$ equals $c(\chi)$.
\end{conj}

This conjecture is known to be true for solvable groups (see \cite{AC}, \cite{FT}). After the late 1980's not much progress was made until recently. In \cite{BKNT}, R. Boltje, A. Kleshchev, G. Navarro, and P. Tiep proved a reduction theorem for this conjecture. This means that there is a condition that if satisfied by all finite non-abelian simple groups, then Feit Conjecture holds (see Definition 3.1 and Theorem A of \cite{BKNT}). In \cite{BKNT}, which is our main reference in this paper, it was proved that the alternating groups and the sporadic groups satisfy the so-called inductive Feit condition. It was also proved that for a prime $p$, and $q=p^n$ the simple groups $\PSL_2(q)$ also satisfy the inductive Feit condition, but when $p$ is odd it was proved under the assumption that $n$ is odd (see Theorem 8.4 of \cite{BKNT}). Here we remove the assumption of $n$ being odd, therefore we get the following.

\begin{thmA}
Let $p$ be a prime and $q=p^n$ so that the group $S=\PSL_2(q)$ is simple. Then $S$ satisfies the inductive Feit condition.
\end{thmA}
    
We also prove the condition for the Suzuki groups.
\begin{thmB}
The Suzuki groups $\Sz(q)$, where $q=2^{2n+1}$ and $n\in \mathbb{Z}_{\geq 1}$, satisfy the inductive Feit condition.
\end{thmB}
The statement of the inductive Feit condition contains a central ordering of $\mathcal{G}$-triples (condition (iii) of Definition \ref{IFC}), where $\mathcal{G}=\Gal(\mathbb{Q}^{ab}/\mathbb{Q})$. This central ordering also appears in other inductive conditions such as in the inductive Galois-McKay condition of \cite{NSV}, to which the Galois-McKay conjecture is reduced (see Definition \ref{IGMC}). It involves finding suitable projective representations. The techniques developed in \cite{BKNT} (see, for example, Theorem 3.2) help linearizing the conditions that need to be satisfied, making inductive conditions of this kind easier to prove. As an example, we have completed the remaining unpublished cases of Theorem C below.

\begin{thmC}
    Let $p$ be a prime and $q=p^n$ so that the group $S=\PSL_2(q)$ is simple. Then $S$ satisfies the inductive Galois-McKay condition for any prime $\ell$ dividing $|S|$.
\end{thmC}
We remark that this result was announced by B. Späth in \cite{NSV}, however, the manuscript has never been made public. We thank B. Späth for allowing us to work on this problem independently. Finally, we refer the reader to the introductions of \cite{BKNT}, \cite{NSV} for more information about the history of the Feit conjecture and the Galois-McKay conjecture.\\

\section{General results}
The goal of this section is to set the notation and study some general results that should always be kept in mind for the subsequent sections. The letters $G,X$ will always denote finite groups, and we will follow notations in \cite{Is},\cite{N1} for characters and (projective) representations. For a more detailed treatment of the next three subsections see Section 2  of \cite{BKNT} and Section 1 of \cite{NSV}.\\

\subsection{Introducing \texorpdfstring{$\mathcal{G}$}{G}-triples.}
Following Chapter 11 of \cite{Is}, if $N$ is a normal subgroup of $G$ and $\theta\in\Irr(N)$ is $G$-invariant, we say that the triplet $(G,N,\theta)$ is a \textbf{character triple}.

We denote by $\mathbb{Q}^{ab}$ the field generated by all roots of unity in $\mathbb{C}$. It is the maximal abelian extension of $\mathbb{Q}$ in the field of algebraic numbers. Let
$\mathcal{G}$ be a subgroup of $ \Gal(\mathbb{Q}^{ab}/\mathbb{Q})$, note that $\mathcal{G}$ is abelian. Note that $\mathcal{G}$ acts on the irreducible characters of any finite group $X$ via $\psi^{\sigma}(x)=\sigma(\psi(x))$ for any $\sigma\in\mathcal{G}, \psi\in\Irr(X), x\in X$. Let us mention a general fact, if $\Gamma$ is a group acting on $X$ from the right via automorphisms, then $\Gamma$ acts naturally on $\Irr(X)$ via $\theta^{\gamma}(x^{\gamma})=\theta(x)$ for any $\gamma\in\Gamma$, $\theta\in\Irr(X)$, $x\in X$. The actions of $\Gamma$ and $\mathcal{G}$ commute so $\Gamma\times\mathcal{G}$ acts naturally on $\Irr(X)$. In particular, if $N$ is normal in $G$, considering the action of the group $\mathcal{G}$ on $\Irr(N)$, we say that for $\theta\in\Irr(N)$, the triplet $(G,N,\theta)$ is a \textbf{$\mathcal{G}$-triple }if the $G$-orbit of $\theta$ is contained in the $\mathcal{G}$-orbit of $\theta$ and we write $(G,N,\theta)_{\mathcal{G}}$. 
\subsection{Review of projective representations.}
Recall that a projective representation of $G$ is a map $\mathcal{P}:G\to \GL_n(\mathbb{C})$ so that for any $g,h\in G$ there exists some $\alpha(g,h)\in\mathbb{C}^\times$ such that $\mathcal{P}(gh)=\alpha(g,h)\mathcal{P}(g)\mathcal{P}(h)$. Therefore $\mathcal{P}$ determines function $\alpha:G\times G  \to \mathbb{C}^\times$ which we call factor set associated to $\mathcal{P}$. It turns out that $\alpha$ satisfies the 2-cocycle relation $\alpha(g,h)\alpha(gh,k)=\alpha(g,hk)\alpha(h,k)$, so $\alpha\in Z^2(G,\mathbb{C}^\times)$, the group 2-cocycles of G with values in $\mathbb{C}^\times$. The notion of similarity of two projective representations is the same as for ordinary representations. We also recall that for any function $\mu:G\to \mathbb{C}^\times$, the function $\mu\mathcal{P}$ is a projective representation with factor set $\delta(\mu)\alpha$ where $\delta(\mu)(g,h)=\mu(g)\mu(h)\mu(gh)^{-1}$. The functions of the form $\delta(\mu)$ form the subgroup $B^2(G,\mathbb{C}^\times)\leq Z^2(G,\mathbb{C}^\times)$ of 2-coboundaries of $G$ with values in $\mathbb{C}^\times$.

When $(G,N,\theta)$ is a character triple, the character $\theta\in\Irr(N)$ may or may not extend to $G$. If it extends we will have a (ordinary) representation of $G$ such that when restricted to $N$ affords $\theta$, but when it does not extend we have to replace the notion of ordinary by projective.
The following is Definition 5.2 from \cite{N1}.
\begin{dfn}\label{def ass}
    Suppose that $(G,N,\theta)$ is a character triple. We say that a projective representation  $$
    \mathcal{P}:G\to \GL_{\theta(1)}(
    \mathbb{C)}$$ is \textbf{associated} with $\theta$ (or with $(G,N,\theta)$ )if\\
    \begin{itemize}
        \item[(a)] $\mathcal{P}_N$ is a representation of $N$ affording $\theta$, and
        \item[(b)] $\mathcal{P}(ng)=\mathcal{P}(n)\mathcal{P}(g)$ and $\mathcal{P}(gn)=\mathcal{P}(g)\mathcal{P}(n)$ for all $g\in G$ and $n\in N$.
    \end{itemize}
\end{dfn}
For any character triple $(G,N,\theta)$, there exists a projective representation $\mathcal{P}$ associated with $\theta$ by Theorem 11.2 of \cite{Is}, and if $\mathcal{P'}$ is another such such projective representation then $\mathcal{P}$ is similar to $\mu\mathcal{P}'$ for some $\mu:G\to \mathbb{C}^\times$ constant on cosets of $N$. Therefore the factor sets of $\mathcal{P},\mathcal{P}'$ differ by a 2-coboundary. The next is Lemma 1.4 of \cite{NSV}.
\begin{lem}\label{1.4 NSV}
    Suppose $N$ is normal in $G$, $\theta\in\Irr(N)$, and assume that $\theta^{g\sigma}=\theta$ for some $g\in G$ and $\sigma\in \Gal(\mathbb{Q}^{ab}/\mathbb{Q})$. Let $\mathcal{P}$ be a projective representation of $G_{\theta}$ associated with $\theta$ with entries in $\mathbb{Q}^{ab}$ and factor set $\alpha\in Z^2(G_{\theta},\mathbb{C}^\times)$. Then 
    $$
    \mathcal{P}^{g\sigma}(x)=\mathcal{P}(gxg^{-1})^{\sigma}
    $$
    where $\sigma$ is applied entry-wise, defines again a projective representation of $G_\theta$ associated with $\theta$ with factor set $\alpha^{g\sigma}$ given by $\alpha^{g\sigma}(x,y)=\alpha^{g}(x,y)^{\sigma}=\alpha(gxg^{-1},gyg^{-1})^{\sigma}$ for $x,y\in G_{\theta}$. In particular, there is a unique function $$
    \mu_{g\sigma}:G_{\theta}\to\mathbb{C}^\times
    $$
    with $\mu_{g\sigma}(1)=1$ and constant on cosets of $N$ such that $\mathcal{P}^{g\sigma}$ is similar to $\mu_{g\sigma}\mathcal{P}$.
    \end{lem}
The assumption of the entries being in $\mathbb{Q}^{ab}$ can always be satisfied by Theorem 2.1 of \cite{BKNT}. This result also allows making the assumption that the values of the factor sets are roots of unity.
    
\subsection{Central orderings and inductive conditions} The next definition is a key element in the inductive conditions that we will study. It first appeared in \cite{NSV} for a particular subgroup of $\Gal(\mathbb{Q}^{ab}/\mathbb{Q})$ and later in \cite{BKNT} (Definition 2.3).
\begin{dfn}[Central order of $\mathcal{G}$-triples]\label{central order}
Suppose that $(G, N, \theta)_{\mathcal{G}}$ and $(H, M, \varphi)_{\mathcal{G}}$ are $\mathcal{G}$-triples. We write
$$
(G, N, \theta)_{\mathcal{G}} \succeq_{\mathbf{c}}(H, M, \varphi)_{\mathcal{G}}
$$
if all the following conditions hold:
\begin{itemize}
    \item [(i)]$G=N H, N \cap H=M$, and $\mathbf{C}_{G}(N) \subseteq H$.
    \item [(ii)] $(H \times \mathcal{G})_{\theta}=(H \times \mathcal{G})_{\varphi}$. In particular, $H_{\theta}=H_{\varphi}$.
    \item [(iii)]There are projective representations $\mathcal{P}$ and $\mathcal{P}^{\prime}$ associated with $\left(G_{\theta}, N, \theta\right)$ and $\left(H_{\varphi}, M, \varphi\right)$ with entries in $\mathbb{Q}^{\text {ab }}$ and factor sets $\alpha$ and $\alpha^{\prime}$, respectively, such that $\alpha$ and $\alpha^{\prime}$ take roots of unity as values, $\alpha_{H_{\theta} \times H_{\theta}}=\alpha^{\prime}$, and, for each $c \in \mathbf{C}_{G}(N)$, the scalar matrices $\mathcal{P}(c)$ and $\mathcal{P}^{\prime}(c)$ are associated with the same scalar $\zeta_{c}$.
    \item [(iv)]For every $a \in(H \times \mathcal{G})_{\theta}$, the functions $\mu_{a}$ and $\mu_{a}^{\prime}$ given by Lemma \ref{1.4 NSV} agree on $H_{\theta}$. 
\end{itemize}

\end{dfn}

From now on we reserve the letter $\mathcal{G}$ for the full group $\Gal(\mathbb{Q}^{ab}/\mathbb{Q})$, $\mathcal{J}$ will denote an arbitrary subgroup of $\mathcal{G}$ of the form $\Gal(\mathbb{Q}^{ab}/F)$ for an intermediate field $\mathbb{Q}\subseteq F\subseteq \mathbb{Q}^{ab}$. We allow $\mathcal{J}$ to be of the following form. Let $\ell$ be an arbitrary but fixed prime. We will denote by $\mathcal{H}$ the subgroup of $\mathcal{G}$ consisting of the $\sigma \in \mathcal{G}$ for which there exists an integer $f$ such that $\sigma(\xi)=\xi^{\ell^f}$ for every root of unity $\xi$ of order not divisible by $\ell$. We now state definitions 3.1 of \cite{NSV} and \cite{BKNT} respectively.

\begin{dfn}\label{IGMC}

Let $S$ be a finite non-abelian simple group, with $\ell$ dividing $|S|$. Let
$X$ be a universal covering group of $S$, $L\in \Syl_{\ell}(X)$ and $\Gamma= \Aut(X)_L $. We say that $S$ satisfies the \textbf{inductive Galois–McKay condition for $\ell$} if there exist some $\Gamma$-stable proper subgroup $U$ of $X$ with $\mathbf{N}_X(L)\subseteq U $ and some $\Gamma \times \mathcal{H}$-equivariant bijection
$$
\Omega: \Irr_{\ell'}(X)\to \Irr_{\ell'}(U)
$$
such that for every $\chi\in \Irr_{\ell'}(X)$ we have 
$$
(X\rtimes \Gamma_{\chi^{\mathcal{H}}}, X, \chi)_{\mathcal{H}} \succeq_{\mathbf{c}}(U\rtimes\Gamma_{\chi^{\mathcal{H}}}, U, \Omega(\chi))_{\mathcal{H}}.
$$
\end{dfn}
\begin{dfn}\label{IFC}

Let $S$ be a finite non-abelian simple group, and let
$X$ be a universal covering group of $S$. We say that $S$ satisfies the \textbf{inductive Feit condition} if for every $\chi\in\Irr(X)$, there exists a pair $(U,\mu)$, where $U<X$ and $\mu\in \Irr(U)$ satisfying the following conditions:
\begin{itemize}
    \item[(i)] $U=\mathbf{N}_X(U)$ and $U$ is \textbf{intravariant} in $X$, i.e. for any $\delta\in \Aut(X)$, $U^\delta=U^x$ for some $x\in X$.
    \item [(ii)] $(\Gamma\times\mathcal{G})_{\chi}=(\Gamma\times\mathcal{G)_{\mu}}$, where $\Gamma=\Aut(X)_U$.
    \item [(iii)] $
(X\rtimes \Gamma_{\chi^{\mathcal{G}}}, X, \chi)_{\mathcal{G}} \succeq_{\mathbf{c}}(U\rtimes\Gamma_{\chi^{\mathcal{G}}}, U, \mu)_{\mathcal{G}}.
$
\end{itemize}

\end{dfn}
In both conditions one has to prove a central ordering. In some sense the inductive Feit condition is harder since the Galois group is bigger and also the condition does not point out where one should look at. The inductive Galois-McKay tells us more precisely what choice of $U$ we should take and since one has to construct the equivariant bijection first, it is more clear what to do. However this lack of freedom can turn the problem to be more difficult as we will see with the characters of $\SL_2(q)$ of degree $(q-1)/2$ below. 
\subsection{Techniques to prove the inductive conditions.}
As said before, in Theorem 3.2 of \cite{BKNT} there are criteria that help substantially when proving the condition involving ``$\succeq_c$". The next theorem is a restatement of this theorem (skipping (c),(e) from there, since we will not use them) in a way that is also applicable to prove the inductive Galois-McKay condition. Moreover, we add one extra statement, part (f), that will be needed in the next section.

\begin{teo}\label{crit}
Let $S$ be a non-abelian finite simple group, let $X$ be a universal covering group of $S$, and let $\chi \in \operatorname{Irr}(X)$. Further, let $U<X$ be self-normalizing and intravariant in $X$. Let $\Gamma=\operatorname{Aut}(X)_{U}$ and assume that $\mu \in \operatorname{Irr}(U)$ is such that $(\Gamma \times \mathcal{J})_{\chi}=(\Gamma \times \mathcal{J})_{\mu}$.\\

\begin{itemize}
    \item [(a)] If $[ \chi_U, \mu]=1$ then $(X\rtimes \Gamma_{\chi^{\mathcal{J}}}, X, \chi)_{\mathcal{J}} \succeq_{\mathbf{c}}(U\rtimes\Gamma_{\chi^{\mathcal{J}}}, U, \mu)_{\mathcal{J}}.$ 
    \item[(b)] If $\Aut(X)_{\chi}=\Inn(X)$ and if $\chi$ and $\mu$ lie over the same irreducible character of $\mathbf{Z}(X)$, then $(X\rtimes \Gamma_{\chi^{\mathcal{J}}}, X, \chi)_{\mathcal{J}} \succeq_{\mathbf{c}}(U\rtimes\Gamma_{\chi^{\mathcal{J}}}, U, \mu)_{\mathcal{J}}.$ 
    \item [(c)] Suppose that $\Sigma$ is a finite group containing $X$ as a normal subgroup such that the conjugation action of $\Delta:=\mathbf{N}_{\Sigma}(U)$ on $X$ induces the subgroup $\Gamma_{\chi^{\mathcal{J}}}$. Suppose that $\chi$ has an extension $\chi^{\sharp}\in \Irr(\Sigma_{\chi})$ and $\mu$ has an extension $\mu^{\sharp}\in \Irr(\Delta_{\mu})$ such that $[ \chi^{\sharp}_{\mathbf{C}_{\Delta_{\mu}}(X)}, \mu^{\sharp}_{\mathbf{C}_{\Delta_{\mu}}(X)}]\neq 0$. Assume further that, for each $(\delta,\sigma)\in (\Gamma \times\mathcal{J})_{\chi}$, the unique linear character $\gamma\in \Irr(\Sigma_{\chi})$ with $(\chi^{\sharp})^{\delta\sigma}=\gamma\chi^{\sharp}$ also satisfies $(\mu^{\sharp})^{\delta\sigma}=\gamma\mu^{\sharp}$ after identifying $\Sigma_{\chi}/X=\Delta_{\mu}/U$. Then $(X\rtimes \Gamma_{\chi^{\mathcal{J}}}, X, \chi)_{\mathcal{J}} \succeq_{\mathbf{c}}(U\rtimes\Gamma_{\chi^{\mathcal{J}}}, U, \mu)_{\mathcal{J}}.$ 
    \item[(d)] Suppose that $\Sigma$ is a finite group containing $X$ as a normal subgroup such that the conjugation action of $\Delta:=\mathbf{N}_{\Sigma}(U)$ on $X$ induces the subgroup $\Gamma_{\chi^{\mathcal{J}}}$. Assume in addition that both $\chi$ and $\mu$ are real-valued, $[\chi_{\mathbf{Z}(X)},\mu_{\mathbf{Z}(X)}]\neq 0$, and $|\Delta_{\mu}/U|$ is odd. Then $(X\rtimes \Gamma_{\chi^{\mathcal{J}}}, X, \chi)_{\mathcal{J}} \succeq_{\mathbf{c}}(U\rtimes\Gamma_{\chi^{\mathcal{J}}}, U, \mu)_{\mathcal{J}}.$ 
    \item[(e)] Suppose that $\Aut(X)_{\chi^{\mathcal{J}}}/\Inn(X)$ is a $p$-group and $[\chi_U,\mu]$ is not divisible by $p$. Then $(X\rtimes \Gamma_{\chi^{\mathcal{J}}}, X, \chi)_{\mathcal{J}} \succeq_{\mathbf{c}}(U\rtimes\Gamma_{\chi^{\mathcal{J}}}, U, \mu)_{\mathcal{J}}.$ 
    \item[(f)] Suppose that $\Sigma$ is a finite group containing $X$ as a normal subgroup such that the conjugation action of $\Delta:=\mathbf{N}_{\Sigma}(U)$ on $X$ induces the subgroup $ \Gamma_{\chi^{\mathcal{J}}}$. Assume that for a prime $p$ the group $\Sigma/X$ has a normal $p$-complement $K/X$ such that $\mathbf{C}_{\Sigma}(X)\subseteq \mathbf{C}_{\Sigma}(K_{\chi})$ and also $\chi$ extends to a character $\chi^*\in \Irr(K_{\chi})$ such that $(\Sigma, K_{\chi}, \chi^*)_{\mathcal{J}}$ is a ${\mathcal{J}}$-triple and $(\Delta \times \mathcal{J})_{\chi}=(\Delta \times \mathcal{J})_{\chi^*}$. On the $U$ side, assume that $\mu$ extends to a character $\mu^*\in \Irr(K \cap \Delta_{\mu})$ such that $(\Delta, K\cap \Delta_{\mu}, \mu^*)_{\mathcal{J}}$ is a ${\mathcal{J}}$-triple and  $(\Delta \times \mathcal{J})_{\mu}=(\Delta \times \mathcal{J})_{\mu^*}$ . Finally, assume that $[ \chi^*_{\Delta_\mu \cap K}, \mu^*]$ is coprime to $p$. Then  $(X\rtimes \Gamma_{\chi^{\mathcal{J}}}, X, \chi)_{\mathcal{J}} \succeq_{\mathbf{c}}(U\rtimes\Gamma_{\chi^{\mathcal{J}}}, U, \mu)_{\mathcal{J}}.$ 
\end{itemize}

\end{teo}
\begin{proof}
 To prove (a)-(e), one substitutes $\mathcal{G}$ by $\mathcal{J}$ in the proof of Theorem 3.2 \cite{BKNT}. We can do this because all the theorems in Section 2 of \cite{BKNT} are true for $\mathcal{J}$. So we just need to prove (f).
First we establish some group theoretical properties about $\Sigma$:
\begin{itemize}
    \item [(i)] $\Sigma=\Delta X$.
    \item [(ii)] $\Delta\cap X=U$.
    \item [(iii)] $\mathbf{C}_{\Sigma}(X) \subseteq \Delta$.
\end{itemize}
Observe that (i) follows since $U$ is intravariant in $X\trianglelefteq \Sigma$. Note that (ii) follows from $U$ being self-normalizing in $X$. To prove (iii), let $g=hx\in \mathbf{C}_{\Sigma}(X)$ with $h\in \Delta$,  $x\in X$, then $U^x=U^{hx}=U$, so $x\in \Delta$, hence $\mathbf{C}_{\Sigma}(X) \subseteq \Delta$. Note that (i) together with $\Delta$ inducing $\Gamma_{\chi^{\mathcal{J}}}$ imply that $(\Sigma, X, \chi)_{\mathcal{J}}$ is a $\mathcal{J}$-triple. Also, since $U$ is normal in $\Delta$ and $\Gamma_{\chi^{\mathcal{J}}}=\Gamma_{\mu^{\mathcal{J}}}$ it is clear that $(\Delta, U, \mu)_{\mathcal{J}}$ is a $\mathcal{J}$-triple. We claim that it is enough to show that  $(\Sigma, X, \chi)_{\mathcal{J}} \succeq_{\mathbf{c}}(\Delta, U, \mu)_{\mathcal{J}}$.\\

To see this, let $\epsilon$: $\Sigma \rightarrow \Aut(X)$ and $\hat{\epsilon}$: $X\rtimes \Gamma_{\chi^{\mathcal{J}}} \rightarrow \Aut(X)$ be defined via the conjugation action on the common normal subgroup X. $ $ Since $\Sigma=\Delta X$ and $\epsilon(\Delta)=\Gamma_{\chi^{\mathcal{J}}}$ it follows that $\epsilon(\Sigma)=\hat{\epsilon}( X\rtimes \Gamma_{\chi^{\mathcal{J}}} )$. We also have that $\hat{\epsilon}( U\rtimes \Gamma_{\chi^{\mathcal{J}}} )=\Gamma_{\chi^{\mathcal{J}}}=\epsilon(\Delta)$, where the first equality comes from the fact that conjugating $X$ by an element in $U$ is an inner automorphism of $X$ leaving $U$ invariant and fixing $\mu$, so $\hat{\epsilon}(U)\subseteq \Gamma_{\mu}= \Gamma_{\chi}\subseteq \Gamma_{\chi^{\mathcal{J}}} $. Note that $ \mathbf{C}_{ X\rtimes \Gamma_{\chi^{\mathcal{J}}}}(X)\subseteq U\rtimes \Gamma_{\chi^{\mathcal{J}}}$, to see this let $x_0\delta \in \mathbf{C}_{ X\rtimes \Gamma_{\chi^{\mathcal{J}}}}(X) $ where $x_0\in X$, $\delta \in \Gamma_{\chi^{\mathcal{J}}}$, we have that for any $x \in X$, $x=x^ {x_0\delta}=(x^{x_0})^{\delta}$, therefore $\delta=\epsilon(x_0)^{-1}$. Since $\delta \in \Gamma$, $U=U^{\delta^{-1}}=U^{x_0}$ so that $x_0\in \mathbf{N}_X(U)=U $.
Therefore, $U\mathbf{C}_{ X\rtimes \Gamma_{\chi^{\mathcal{J}}}}(X) \subseteq  U\rtimes \Gamma_{\chi^{\mathcal{J}}}  $ and all the hypothesis of the butterfly Theorem 2.11 from \cite{BKNT} are satisfied so 
$\left(X \rtimes \Gamma_{\chi^{\mathcal{J}}}, X, \chi\right)_{\mathcal{J}} \succeq_{\mathbf{c}}\left(U \rtimes \Gamma_{\chi^{\mathcal{J}}}, U, \mu\right)_{\mathcal{J}}$ as wanted.\\

To prove $(\Sigma, X, \chi)_{\mathcal{J}} \succeq_{\mathbf{c}}(\Delta, U, \mu)_{\mathcal{J}}$ we just need to focus on conditions (iii), (iv) of Definition \ref{central order}. This is because the first condition of the definition was established at the beginning of the proof, and the second condition follows immediately from the fact that $\Delta$ induces $\Gamma_{\chi^{\mathcal{J}}}$ when acting on $X$ and the fact that $(\Gamma \times \mathcal{J})_{\chi}=(\Gamma \times \mathcal{J})_{\mu}$. Let us now consider the $\mathcal{J}$-triples  $(\Sigma, K_{\chi}, \chi^*)_{\mathcal{J}}$ and $(\Delta, K\cap \Delta_{\mu}, \mu^*)_{\mathcal{J}}$, we want to apply Theorem 2.6 from \cite{BKNT}. Note that $\Sigma=K_{\chi}\Delta$, $K_{\chi}\cap \Delta=K\cap \Delta_{\chi}=K\cap \Delta_{\mu}$ and also $\mathbf{C}_{\Sigma}(K_{\chi})\subseteq \mathbf{C}_{\Sigma}(X) \subseteq \Delta $. Observe that $\Sigma_{\chi*}=\Sigma_{\chi}$, this is because $\Sigma=X\Delta$ and $(\Delta \times \mathcal{J})_{\chi}=(\Delta \times \mathcal{J})_{\chi^*}$. Moreover, note that $\Sigma/K$ is a $p$-group and $\Sigma_{\chi}\cap K=K_{\chi}$. Therefore $\Sigma_{\chi^*}/K_{\chi}=\Sigma_{\chi}/K_{\chi}$ is a $p$-group. Finally, since $(\Delta \times \mathcal{J})_{\chi^*}=(\Delta \times \mathcal{J})_{\chi}=(\Delta \times \mathcal{J})_{\mu}=(\Delta \times \mathcal{J})_{\mu^*}$ we conclude by Theorem 2.6 \cite{BKNT} that $(\Sigma, K_{\chi}, \chi^*)_{\mathcal{J}} \succeq_{\mathbf{c}}(\Delta, K\cap \Delta_{\mu}, \mu^*)_{\mathcal{J}}$. \\

By definition, we obtain two projective representations $\mathcal{P}$ and $\mathcal{P}^{\prime}$ associated with $(\Sigma_{\chi},K_{\chi},\chi^*)$, $(\Delta_{\mu}, K\cap \Delta_{\mu}, \mu^*)$ respectively, satisfying (iii), (iv) from Definition \ref{central order}. We claim that we can use this projective representations to conclude that  $(\Sigma, X, \chi)_{\mathcal{J}} \succeq_{\mathbf{c}}(\Delta, U, \mu)_{\mathcal{J}}$. Indeed, the fact that $\mathcal{P}$ is associated with $(\Sigma_{\chi},K_{\chi},\chi^*)$ automatically gives that $\mathcal{P}$ is associated with $(\Sigma_{\chi},X,\chi)$, since $\chi^*$ extends $\chi$. Similarly, since  $\mathcal{P'}$ is associated with $(\Delta_{\mu}, K\cap \Delta_{\mu}, \mu^*)$ we easily see that is also associated with $(\Delta_{\mu}, U, \mu)$. The remaining requirements from (iii), (iv) of Definition \ref{central order} follow easily from $\Delta_{\chi}=\Delta_{\chi^*}$, $\mathbf{C}_{\Sigma}(X)\subseteq \mathbf{C}_{\Sigma}(K_{\chi})$ and $(\Delta \times \mathcal{J})_{\chi}=(\Delta \times \mathcal{J})_{\chi^*}$. 
\end{proof}
\subsection{Number theoretic prerequisites} When proving that we have an $\mathcal{H}$-equivariant bijection we will make use of some elementary number theory, so one needs to recall the Gauss Sum Formula, namely, let $p>2$ be a prime, fix a choice of $i=\sqrt{-1}$ and let $\omega=e^{2\pi i/p}$ then 
    $$
    \sum_{k=0}^{p-1}\left(\dfrac{k}{p}\right) \omega^k= \begin{cases} 
      \sqrt{p} & \textrm{if   }  p\equiv 1\mod 4 \\
      
      i\sqrt{p} & \text{if   }  p \equiv 3 \mod 4 
   \end{cases}
    $$
If $n\in \mathbb{Z}$, and $p$ does not divide $n$, then 
$$\sum_{k=0}^{p-1}\left(\dfrac{k}{p}\right) \omega^{nk}=\left(\dfrac{n}{p}\right)\sum_{k=0}^{p-1}\left(\dfrac{k}{p}\right) \omega^k.$$
Also, recall the Quadratic reciprocity law, namely, if we let $p,q$ be distinct odd primes, then $\left(\dfrac{q}{p}\right)\left(\dfrac{p}{q}\right)=(-1)^{\frac{p-1}{2}\frac{q-1}{2}}$. For proofs, see for example \cite{Lang}.\\

\subsection{Some facts about \texorpdfstring{$\PSL_2(q)$}{PSL(2,q)}.}
Here we will recall some basic facts about groups of the form $\PSL_2(q)$. Let $p$ be a prime number and let $n$ be a positive integer. The group $S=\PSL_2(q)$ is simple for all such $q$ except $q=2,3$, so we ignore these cases. Also, recall that $\PSL_2(5)\cong \PSL_2(4)$ and $\PSL_2(9)\cong A_6$, the alternating group of order $360$.\\

If $q$ is even, then $\PSL_2(q)=\SL_2(q)$ and the Schur multiplier $M(S)=1$ for $q\neq 4$ and for $q=4$ we have that $M(\PSL_2(4))=C_2$. Therefore, if $X$ is the universal covering group of $S$, then, $X=S$ for $q\neq 4$ and $X=\SL_2(5)$ for $q=4$. If $q$ is odd, we have that $M(S)=C_2$ and $X=\SL_2(q)$ as long as $q\neq9$. For $q=9$, $M(A_6)=C_6$, the cyclic group of order $6$ so $X=6.A_6$.\\

Our notation for irreducible characters of $\SL_2(q)$ will be as listed in Theorems 38.1 and 38.2 of \cite{D} (depending on whether $q$ is even or odd). Recall that $|\SL_2(q)|=q(q^2-1)$ and for $q$ odd, $|\PSL_2(q)|=\frac{q(q^2-1)}{2}$. When $q$ is even, $\Aut(X)=\langle S,\alpha\rangle$ where $\alpha$ acts by taking power 2 to the entries of the matrices. Recall that $\SL_2(q)\cong\SU_2(q)\cong\Sp_2(q)$. We will make use of these isomorphisms in our proofs. One advantage of $\SL_2(q)$ is that it is very explicit. It is convenient for us to explain 3 setups that will appear later on when dealing with $\PSL_2(q)$.

-\textbf{ Setup 1:} Here we consider $X=\SL_2(q)$, then $\Aut(\SL_2(q))=\langle S, \alpha,\tau\rangle$, where $\alpha$ is a field automorphism and $\tau$ is a diagonal automorphism. Fix $\mu$ so that $\mathbb{F}_q^*=\langle\mu\rangle$, the actions can be taken as follows
$$
\begin{pmatrix}
a & b\\
c & d 
\end{pmatrix}^{\alpha}
=\begin{pmatrix}
a^p& b^p\\
c^p & d^p
\end{pmatrix}, \space \begin{pmatrix}
a & b\\
c & d 
\end{pmatrix}^{\tau}
=\begin{pmatrix}
a& b\mu^{-1}\\
c\mu  & d
\end{pmatrix}$$
It is convenient to note the following. Forget the notation for a moment, let $A$ be the group of semilinear transformations on a two dimensional vector space over $\mathbb{F}_q$. We have that $A=\SL_2(q)\rtimes\langle\tau\rangle\rtimes\langle\alpha\rangle$, where $\tau=\begin{pmatrix}
\mu & 0\\
0 & 1
\end{pmatrix}\in \GL_2(q)$, $o(\tau)=q-1$ and $\alpha$ is the semilinear transformation acting via $(\lambda_1 e_1 +\lambda_2 e_2) \cdot \alpha= \lambda_1^p e_1 +\lambda_2^p e_2 $ where $\{e_1,e_2\}$ is the canonical basis. Note that $o(\alpha)=n$. Also, $\GL_2(q)=\SL_2(q)\rtimes\langle\tau\rangle$. Denoting with an overbar reduction modulo scalars, $\bar{A}=\PGL_2(q)\rtimes\langle\bar{\alpha}\rangle$, $\PGL_2(q)=\langle \PSL_2(q),\bar{\tau}\rangle$, $|\PGL_2(q):\PSL_2(q)|=2$ and $o(\bar{\alpha})=n$. The natural map $\bar{A}\to \Aut(\SL_2(q))$ is an isomorphism  and what we previously denoted by $\tau,\alpha\in \Aut(\SL_2(q))$ are precisely the images of $\bar{\tau},\bar{\alpha}\in \bar{A}$. This explains the notation (for details see Theorem 2.1.4 and section 2.2 of \cite{KL}).
\\ We now describe a self-normalizing in $\SL_2(q)$ and intravariant subgroup. Consider the diagonal torus $T=C_{q-1}$, an explicit way to see it is via the matrix
$$
a=\begin{pmatrix}
\mu & 0\\
0 & \mu^{-1}
\end{pmatrix}\in \SL_2(q)
$$
Note that it is unique up to conjugacy. Let $\ell$ be an odd prime dividing $q-1$, we can consider $L\in\Syl_{\ell}(\SL_2(q))$ such that $L\subseteq T$. By Corollary 1.3.2 of \cite{Bon}, we have that $U:=\mathbf{N}_{\SL_2(q)}(T)=\langle a,u\rangle$ where 
$$
u=\begin{pmatrix}
0 & 1\\
-1 & 0
\end{pmatrix}\in \SL_2(q).
$$
Note that $u^2=a^{(q-1)/2}$, $a^u=a^{-1}$ so in fact $U$ is dicyclic of degree $(q-1)/2$. We will give a short review on dicyclic groups shortly. Also, by Theorem 1.4.3 of \cite{Bon}, $U=\mathbf{N}_{\SL_2(q)}(L)$, hence it is self-normalizing and intravariant in $\SL_2(q)$. In fact it is invariant under $\alpha,\tau\in \Aut(\SL_2(q))$, using the notation we just introduced. Note that in fact, $\tau,\alpha$ fix $L$, hence $\Aut(X)_L=\Aut(X)_U=\langle\bar{U},\tau,\alpha\rangle$. Define $M:=\{y\in\GL_2(q): \det(y)=\pm1\}$, let $U^{\sharp}=\mathbf{N}_M(U)$, and note that $U^{\sharp}$ is semidirect product of $U$ with the subgroup generated by 
$$
\begin{pmatrix}
1 & 0\\
0 & -1
\end{pmatrix}.
$$

When $q\equiv3$ mod $4$, we have that $\mathbf{Z}(X)=\mathbf{Z}(M)=\langle z \rangle$ and $M/\mathbf{Z}(M)\cong\PGL_2(q)$ therefore $\tau$ is induced by some element in $U^{\sharp}$. \\

 -\textbf{ Setup 2:} Here $X=\SU_2(q)$, i.e. we endow a 2 dimensional vector space $V$ over $\mathbb{F}_ {q^2}$ with a unitary form $\kappa$ and we consider the determinant one transformations preserving this form. Fix an orthonormal basis to see elements as matrices. We naturally write $\Aut(\SU_2(q))=\langle \PSU_2(q), \alpha,\tau\rangle$, where $\alpha$ is a field automorphism and $\tau$ is a diagonal automorphism. Fix $\rho$ so that $\mathbb{F}_{q^2}^{\times}=\langle\rho\rangle$, the actions can be seen as follows
$$
\begin{pmatrix}
a & b\\
c & d 
\end{pmatrix}^{\alpha}
=\begin{pmatrix}
a^p& b^p\\
c^p & d^p
\end{pmatrix}, \space \begin{pmatrix}
a & b\\
c & d 
\end{pmatrix}^{\tau}
=\begin{pmatrix}
a& b\rho^{1-q}\\
c\rho^{q-1} & d
\end{pmatrix}$$
Again we forget notation for a moment, let $A$ be the group of semilinear transformations, say $g$, on $V$ such that for all $v_1,v_2\in V$, $\kappa (v_1\cdot g,v_2\cdot g)=\lambda \kappa(v)^\sigma$ for some $\sigma\in\Aut(\mathbb{
F}_{q^2})$, and $\lambda\in \mathbb{F}_{q^2}$. We have that $\GU_2(q)=\SU_2(q)\rtimes\langle\tau\rangle$ 
 and $A=\GU_2(q)\mathbb{F}_{q^2}^{\times}\rtimes\langle\alpha\rangle$, where, $\tau=\begin{pmatrix}
\rho^{q-1} & 0\\
0 & 1
\end{pmatrix}\in \GU_2(q)$, $o(\tau)=q+1$ and $\alpha$ is the semilinear transformation acting via $(\lambda_1 e +\lambda_2 f) \cdot \alpha= \lambda_1^p e +\lambda_2^p f $ where $\{e,f\}$ is the fixed orthonormal basis, note that $o(\alpha)=2n$. Denoting with an overbar reduction modulo scalars, $\bar{A}=\PGU_2(q)\rtimes\langle\bar{\alpha}\rangle$, $|\PGU_2(q):\PSU_2(q)|=2$ and $o(\bar{\alpha})=2n$ (see section 2.3 of \cite{KL}). The natural map $\bar{A}\to \Aut(\SU_2(q))$ is a surjective homomorphism with kernel $C_2$ and what we have denoted by $\tau,\alpha\in \Aut(\SU_2(q))$ are precisely the images of $\tau,\alpha\in A$ ($u$, defined below, acts on $\SU_2(q)$ as $\alpha^n$, so $\PGU_2(q)$ injects into $\Aut(\SU_2(q))$, also, the kernel of the above map is generated by the image of $\alpha^nu^{-1}$). In this set up we will look at the non-split torus, namely we want to consider $T=\langle b\rangle$ the unique up to conjugacy $C_{q+1}$ subgroup of $\SU_2(q)$. Let $\ell$ be a prime dividing $q+1$, we take $L\in \Syl_{\ell}(X), L\subseteq T$. Let $U=\mathbf{N}_X(T)=\langle b,u\rangle$ 
 $$
 b=\begin{pmatrix}
\rho^{q-1} & 0\\
0 & \rho^{1-q}
\end{pmatrix},u=\begin{pmatrix}
0 & 1\\
-1 & 0
\end{pmatrix}\in \SU_2(q)
 $$
 
Note that $b^{q+1}=1,$ $u^2=b^{(q+1)/2}$, $b^u=b^{-1}$ so $U'$ is dicyclic of degree $(q+1)/2$.  By Theorem 1.4.3 of \cite{Bon}, $U=\mathbf{N}_{\SU_2(q)}(L)$, hence it is self-normalizing and intravariant in $\SU_2(q)$. It is invariant under $\alpha,\tau\in\Aut(\SU_2(q))$. Similarly, $\Aut(X)_L=\Aut(X)_{U}=\langle \bar{U},\tau,\alpha\rangle$. Let $M:=\{ y\in \GU_2(q): \det(y)= \pm 1\}$, and define $U^{\sharp}=\mathbf{N}_M(U)$. Note that $U^{\sharp}$ is semidirect product of $U$ with the subgroup generated by 
$$
\begin{pmatrix}
1 & 0\\
0 & -1
\end{pmatrix}\in M.
$$

When $q\equiv1$ mod $4$, we have that $\mathbf{Z}(X)=\mathbf{Z}(M)=\langle z \rangle$ and $M/\mathbf{Z}(M)\cong\PGU_2(q)$ therefore $\tau$ is induced by some element in $U^{\sharp}$. \\\\

-\textbf{ Setup 3:} Here $X=\SU_2(q)$, i.e. we endow a 2 dimensional vector space $V$ over $\mathbb{F}_ {q^2}$ with a unitary form $\kappa$ and we consider the determinant one transformations preserving this form. Fix a unitary basis, i.e., a basis $\{e,f\}$ such that $\kappa(e,e)=\kappa(f,f)=1,$ $\kappa(e,f)=1$,  to see elements as matrices. We naturally write $\Aut(\SU_2(q))=\langle \PSU_2(q), \alpha,\tau\rangle$, where $\alpha$ is a field automorphism and $\tau$ is a diagonal automorphism. Fix $\rho$ so that $\mathbb{F}_{q^2}^{\times}=\langle\rho\rangle$, the actions can be seen as follows
$$
\begin{pmatrix}
a & b\\
c & d 
\end{pmatrix}^{\alpha}
=\begin{pmatrix}
a^p& b^p\\
c^p & d^p
\end{pmatrix}, \space \begin{pmatrix}
a & b\\
c & d 
\end{pmatrix}^{\tau}
=\begin{pmatrix}
a& b\rho^{q+1}\\
c\rho^{-q-1} & d
\end{pmatrix}$$
Again we forget notation for a moment, let $A$ be the group of semilinear transformations, say $g$, on $V$ such that for all $v_1,v_2\in V$, $\kappa (v_1\cdot g,v_2\cdot g)=\lambda \kappa(v)^\sigma$ for some $\sigma\in\Aut(\mathbb{
F}_{q^2})$, and $\lambda\in \mathbb{F}_{q^2}$. We have that $\GU_2(q)=\SU_2(q).\langle\tau\rangle$ 
 and $A=\GU_2(q)\mathbb{F}_{q^2}^{\times}\rtimes\langle\alpha\rangle$, where, $\tau=\begin{pmatrix}
\rho^{-1} & 0\\
0 & \rho^q
\end{pmatrix}\in \GU_2(q)$ and $\alpha$ is the semilinear transformation acting via $(\lambda_1 e +\lambda_2 f) \cdot \alpha= \lambda_1^p e +\lambda_2^p f $, note that $o(\alpha)=2n$. Denoting with an overbar reduction modulo scalars, $\bar{A}=\PGU_2(q)\rtimes\langle\bar{\alpha}\rangle$, $|\PGU_2(q):\PSU_2(q)|=2$ and $o(\bar{\alpha})=2n$. The natural map $\bar{A}\to \Aut(\SU_2(q))$ is a surjective homomorphism with kernel $C_2$ given below by an element which is not in $\langle S,\tau\rangle=\PGU_2(q)$ (which injects by the above map into $\Aut(\SU_2(q))$) and what we have denoted by $\tau,\alpha\in \Aut(\SU_2(q))$ are precisely the images of $\tau,\alpha\in A$.
In this set up we will look at the split torus, namely we want to consider $T=\langle a \rangle$ the unique up to conjugacy $C_{q-1}$ subgroup of $\SU_2(q)$. Let $\ell$ be a prime dividing $q-1$, we take $L\in \Syl_{\ell}(X), L\subseteq T$. Let $U=\mathbf{N}_X(T)=\langle a,u\rangle$ 
 $$
 a=\begin{pmatrix}
\rho^{q+1} & 0\\
0 & \rho^{-(q+1)}
\end{pmatrix},u=\begin{pmatrix}
0 & \rho^{\frac{q+1}{2}}\\
-\rho^{\frac{-(q+1)}{2}} & 0
\end{pmatrix}\in \SU_2(q)
 $$
 
Note that $a^{q-1}=1,$ $u^2=a^{(q-1)/2}$, $a^u=a^{-1}$ so $U$ is dicyclic of degree $(q-1)/2$.  By Theorem 1.4.3 of \cite{Bon}, $U=\mathbf{N}_{\SU_2(q)}(L)$, hence it is self-normalizing and intravariant in $\SU_2(q)$. It is invariant under $\alpha,\tau\in\Aut(\SU_2(q))$. Similarly, $\Aut(X)_L=\Aut(X)_{U}=\langle \bar{U},\tau,\alpha\rangle$. Let $M:=\{ y\in \GU_2(q): \det(y)= \pm 1\}$, and define $U^{\sharp}=\mathbf{N}_M(U)$. Note that $|U^\sharp:U|=2$.
When $q\equiv1$ mod $4$, we have that $\mathbf{Z}(X)=\mathbf{Z}(M)=\langle z \rangle$ and $M/\mathbf{Z}(M)\cong\PGU_2(q)$ therefore $\tau$ is induced by some element in $U^{\sharp}$. Finally let 
$$
t=\begin{pmatrix}
-\rho^{\frac{q^2-1}{4}} & 0\\
0 & \rho^{\frac{q^2-1}{4}}
\end{pmatrix}\in \SU_2(q).
$$
One can check explicitly that $t$ acts by conjugation on $\SU_2(q)$ exactly as $\alpha^n$ therefore $t^{-1}\alpha^n$ is in the kernel of $A\to \Aut(\SU_2(q))$, moreover commutes with $\alpha$ and its projection to $\bar{A}$ is not in $\PGU_2(q)$.\\\\

\subsection{Representations of dicyclic groups.}
Since these groups will appear often we give a quick reminder of its characters. A finite group $G$ is said to be \textbf{dicyclic} of degree $n\geq 2$ if it has 2 generators $x,u$ with orders $2n,4$ respectively such that $u^2=x^n$ and $x^u=x^{-1}$. It has order $4n$, $\mathbf{Z}(G)=\{1,u^2\}$ and it has $n+3$ conjugacy classes. A complete set of representatives of the conjugacy classes is $\{1,x^l,u,u^2,xu\}$ where $1\leq l \leq n-1$. Note that the order of $u,xu$ is $4$. The group $G$ has 4 linear characters which we will denote by $\{1_G, \Sgn_G
, \mu^+, \mu^-\}$, where $1_G,\Sgn_G
$ extend the trivial character of $\langle x\rangle $ with $\Sgn_G(u)=-1$ and $\mu^+(x^l)=(-1)^l=\mu^-(x^l)$. If $n$ is even $\mu^+(u)=1$, $\mu^-(u)=-1$ and if $n$ is odd, $\mu^+(u)=i$, and $\mu^-(u)=-i$. The $n-1$ remaining characters of $G$ are of degree 2 (induced from $\langle
x\rangle
$), we denote them by $\mu_j$ and $\mu_j(x^l)=\xi^{jl}+\xi^{-jl}$ where $1\leq j\leq n-1$ and $\xi$ is a complex root of unity of order $2n$. Observe that when $n$ is odd $\Sgn$ can be distinguished from $\mu^+,\mu^-$ by its field of values and if $n\geq 4$ is even, $\Sgn$ is the unique linear character taking the value $-1$ exactly twice. 
 \\

\section{Inductive Feit condition}
For the proof of the next result, part of the ingredients that we need are extracted from Theorem 8.4 of \cite{BKNT}. The main new ingredient is Theorem \ref{crit}(f).
\begin{teo}
Let $p>2$ be any prime, $q=p^n\geq5$. Then the simple group $S=\PSL_2(q)$ satisfies the inductive Feit condition.
\end{teo}
\begin{proof}
Since $\PSL_2(5)\cong A_5$, $\PSL_2(9)\cong A_6$, by Theorem 7.17 \cite{BKNT}, we assume that $q=7$ or $q\geq 11$. Also, by Theorem 8.4 from \cite{BKNT} we may assume that $n$ is even so $q\equiv 1$ (mod 4). The universal covering group of $S$ is $X=\SL_2(q)$, we view it as $X\cong \SU_2(q)$ and we work as in setup 3. By Theorem 8.4 \cite{BKNT}, we only need to deal with the real characters $\chi_j$ of degree $q+1$ (see \cite{D}). We have that $S\cong \PSU_2(q)$ and also $\Aut(X)=
\langle S,\tau, \alpha \rangle$ where $\langle S,\tau\rangle \cong \PGU_2(q) $ and $\alpha$ acts by taking the $p$-th power to the matrix entries after fixing a unitary basis.\\

Fix $\chi=\chi_j$ of degree $q+1$. Consider the subgroup $U<X$ defined in setup 3, it is self-normalizing in X, intravariant, and $\langle \tau ,\alpha \rangle$-invariant. By the proof of Theorem 8.4 \cite{BKNT} there is a real character $\mu \in \Irr(U)$ such that $\mu(a)=\chi(a)$ and $[ \chi_U, \mu]=3$. Moreover, if we call $\Gamma=\operatorname{Aut}(X)_{U}=\langle \bar{U},\tau,\alpha\rangle $ then $(\Gamma \times \mathcal{G})_{\chi}=(\Gamma \times \mathcal{G})_{\mu}$ and
 $\Gamma_{\chi^{\mathcal{G}}}=\Gamma$. Consider $M=\{ y\in \GU_2(q)$ : $\det(y)=\pm 1 \}$, calling $U^{\sharp}=\mathbf{N}_L(U)$ , we have that $\tau$ is induced by an element in $U^{\sharp}$ . We now consider $\alpha$ as a semilinear transformation acting via $(\lambda_1 e +\lambda_2 f) \cdot \alpha= \lambda_1^p e +\lambda_2^p f $ where $\{e,f\}$ is our fixed unitary basis. We form the group $\Sigma=M \rtimes \langle \alpha \rangle \leq A$, where $A$ is as introduced in the setup 3. This way, $X$ is a normal subgroup of $\Sigma$, and the natural map $\epsilon: \Sigma \rightarrow \Aut(X)$ is surjective with $\ker (\epsilon)=\mathbf{C}_{\Sigma}(X)$ of order 4, because $|\Sigma|=4n|\SL_2(q)|$ and $|\Aut(X)|=n|\SL_2(q)|$ (see setup 3 and Table 2.1 from \cite{KL}). Let $\Delta=\mathbf{N}_{\Sigma}(U)=U^{\sharp}\rtimes \langle \alpha \rangle$. Clearly $\epsilon(\Delta)=\Gamma=\Gamma_{\chi^{\mathcal{G}}}$. Recall that $|M:X|=2$. Since $\Sigma/X=C_2 \rtimes C_{2n}$ is abelian we have that $\Sigma/X$ has a normal $2$-complement $K/X$. From the discussion in setup 3 it follows that $t^{-1}\alpha^n\in \mathbf{C}_{\Sigma}(X)
 $ and we know it commutes with $\alpha$ so it follows that $\mathbf{C}_{\Sigma}(X)\subseteq \mathbf{C}_{\Sigma}(K_{\chi})$ (of course $\{\pm 1\}\subseteq\mathbf{C}_{\Sigma}(X)\subseteq\mathbf{C}_{\Sigma}(K_{\chi})$). Since $|K/X|$ is odd, by Lemma 2.1 \cite{NT}, there is a unique real extension $\chi^*$ of $\chi$ to $K_{\chi}$, this canonical extension satisfies $\mathbb{Q}(\chi^*)=\mathbb{Q(\chi)}$. We claim that $(\Sigma, K_{\chi}, \chi^*)_{\mathcal{G}}$ is a ${\mathcal{G}}$-triple and $(\Delta \times \mathcal{G})_{\chi}=(\Delta \times \mathcal{G})_{\chi^*}$. \\

 To prove the first statement, start by observing that $(\Sigma, X,\chi)_{\mathcal{G}}$ is a $\mathcal{G}$-triple (see the first part of the proof of Theorem \ref{crit}). Now note that since $\Sigma/X$ is abelian, $K_{\chi}$ is normal in $\Sigma$. Let $g\in \Sigma$ and note that $(\chi^*)^g$ is a real character extending $\chi^g=\chi^{\sigma}$ for some $\sigma\in\mathcal{G}$. Since $\sigma$ descends to a Galois automorphism of a cyclotomic field, it commutes with complex conjugation, so it preserves reality. Thus, $\mathbb{Q}(\chi)= \mathbb{Q} (\chi^*)$ is mapped to the reals by $\sigma$. We can consider $(\chi^*)^{\sigma}\in \Irr(K_{\chi})$ and by uniqueness of real extension it follows that $(\chi^*)^g=(\chi^*)^{\sigma}$. It is now clear that $(\Sigma, K_{\chi}, \chi^*)_{\mathcal{G}}$ is a ${\mathcal{G}}$-triple. The fact that $(\Delta \times \mathcal{G})_{\chi^*}\subseteq (\Delta \times \mathcal{G})_{\chi} $ is clear since $\chi^*$ extends $\chi$. Now assume that $\chi^{(\delta,\sigma)}=\chi$ for some $(\delta,\sigma)\in (\Delta \times \mathcal{G}) $, then, as before, the character $(\chi^*)^{(\delta,\sigma)}\in \Irr(K_{\chi})$ is a real and extends $\chi$ so it must be equal to $\chi^*$.\\

 Recall that $\Delta$ induces $\Gamma=\Gamma_{\chi^{\mathcal{G}}}=\Gamma_{\mu^{\mathcal{G}}}$ so clearly $(\Delta, U,\mu)_{\mathcal{G}}$ is a $\mathcal{G}$-triple. From  the fact that $\Sigma/X\cong \Delta/U$, is easy to see that $K\cap \Delta/U$ is a normal $2$-complement of $\Delta/U$ so that $|K\cap \Delta_{\mu}/U|$ is odd. Thus $\mu$ has a canonical real extension $\mu^*\in \Irr(K \cap \Delta_{\mu})$. With the same arguments as before we easily see that $(\Delta, K\cap \Delta_{\mu}, \mu^*)_{\mathcal{G}}$ is a ${\mathcal{G}}$-triple and  $(\Delta \times \mathcal{G})_{\mu}=(\Delta \times \mathcal{G})_{\mu^*}$. \\

 Finally, we argue that $l=[ \chi^*_{\Delta_\mu \cap K}, \mu^*]$ is odd. Note that $ \chi^*_{\Delta_\mu \cap K}$ restricted to $U$ is $\chi_U$ and that $ (\mu^*)_U=\mu$. Therefore $l\leq 3$. By Corollary 2.2 of \cite{NT}, $\mu^*$ is the unique real character in $\Irr(\Delta_\mu \cap K|\mu)$. Assume that $l=2$, then since $[ \chi_U, \mu]=3$ there exist $\theta\in \Irr(\Delta_\mu \cap K)-\{\mu^*\}$ extending $\mu$ such that $[ \chi^*_{\Delta_\mu \cap K}, \theta]=1$. The fact that $ \chi^*_{\Delta_\mu \cap K}$ is real forces the complex conjugate $\Bar{\theta}\in\Irr(\Delta_\mu \cap K) $ to also be a constituent of  $\chi^*_{\Delta_\mu \cap K}$. Therefore, since $\theta \neq \Bar{\theta}$ and $\bar{\theta}$ also lies over $\mu$ (by reality of $\mu$), we have that $[ \chi_U, \mu]>3$. So $l\neq 2$. Assume now that $l=0$, then, we have that all irreducible constituents of $\chi^*_{\Delta_\mu \cap K}$ lying over $\mu$ are complex valued. This is a contradiction. Note that this forces us to have exactly $2$ constituents of  $\chi^*_{\Delta_\mu \cap K}$ lying over $\mu$, say $\theta,\bar{\theta}\in \Irr(\Delta_\mu \cap K|\mu)$ and this would force $[ \chi_U, \mu]$ to be even (note that $\mu$ appears with the same multplicity in both $\theta_U,\bar{\theta}_U$). Hence $l$ is odd (in fact is clear that it is precisely $3$). Now we can apply Theorem \ref{crit}(f) to conclude that $(X,\chi)$ satisfies the inductive Feit condition. 
\end{proof}

This, combined with Theorem $8.1$ from \cite{BKNT} implies Theorem A. We now prove the inductive Feit condition for the Suzuki groups $^{2}B_2(2^{2n+1})=\Sz(q)$ where $q=2^{2n+1}$  and $n\geq 1$. 
We first deal with the Suzuki groups that have trivial Schur multiplier.
\begin{teo}\label{suzuki trivial schur}
    Let $q=2^{2n+1}$ with $n\in \mathbb{Z}_{\geq2}$. Then the simple group $S=\Sz(q)$ satisfies the inductive Feit condition.
\end{teo}
\begin{proof}
Since $S$ has trivial Schur multiplier, we have that $X=S$ (see Theorem 5.1.4 of \cite{KL}). We will follow the notation of Section 1 of \cite{Bur} for the irreducible characters of $X$, and also we will use facts from that section without any further reference. It is well known that $\Aut(X)=\Inn(X)\rtimes \langle \alpha \rangle $ where $\Inn(X) \cong X$ via conjugation and $\alpha$ is the field automorphism (see Theorem 11 from \cite{Suz}), so it has order $2n+1$.  \\

$(a)$ First, we consider the characters $1,\Pi, \Gamma_1, \Gamma_2$ of X. Let $U=\mathbf{N}_X(P)$ where $P\in \Syl_2(X)$, so $U$ is of self-normalizing and intravariant in $X$. The group $U$ has order $q^2(q-1)$ and $|X:U|=q^2+1$. We now mention a few facts about the representation theory of $U$, for a proof of this we refer the reader to sections 11, 17 of \cite{Suz}. The group $U$ has exactly $q-1$ linear characters and $3$ nonlinear characters $\varrho_1, \varrho_2, \varrho_3$. The degree of $\varrho_1, \varrho_2$ is $r(q-1)/2$ where $r^2=2q$ while the degree of $\varrho_3$ is $q-1$. Up to a change of notation, we have that $(\Gamma_1)_U=\varrho_1$ and  $(\Gamma_2)_U=\varrho_2$ and of course $\varrho_3$ is rational since it is the unique irreducible character of its degree. We let $\Gamma=\Aut(X)_U$.\\
If $\chi=1$, then we take $\mu=1_U$. Both are rational and $\Gamma$-stable, also $\left[\chi_{U}, \mu\right]=1$ so we are done by Theorem \ref{crit}(a). Consider now the Steinberg character $\chi=\Pi$, then we again take $\mu=1_U$. Similarly, $\chi,\mu$ are rational and since $\chi$ is the unique character of degree $q^2$, both are also $\Gamma$-stable. Since $(1_U)^X=1+\Pi$ we are also done by Theorem \ref{crit}(a).\\
Now consider $\chi=\Gamma_i$ for $i=1,2$. Let $\mu=\varrho_i$ so that $\chi_U=\mu$. Since $\chi$ is the unique extension of $\mu$ to $X$ it follows that $(\Gamma\times \mathcal{G})_{\chi}=(\Gamma\times\mathcal{G})_{\mu}$ so we are also done by Theorem \ref{crit}(a).\\

All the remaining irreducible characters of $X$ are Deligne-Lusztig characters up to a sign. To deal with these, one could use Proposition 8.5 from \cite{BKNT}, but here we present a more elementary treatment. In fact the idea we use is the same but without referring to Deligne-Lusztig theory. Note that in the notation of \cite{Bur}, the elements $x,y,z$ are real, so they are conjugate to their inverses.\\

$(b)$ Next, we consider the characters $\Omega_s$ for $1\leq s\leq (q-2)/2$ of degree $q^2+1$. Let $H=\langle x \rangle$, we have that $H$ is a cyclic Hall subgroup of order $q-1$. Let $U=\mathbf{N}_X(H)$, then $|U:H|=2$ and $U=\mathbf{N}_X(L)$ for every nontrivial subgroup $L$ of $H$ (see section 1 of \cite{Bur}). In particular, $U$ is the normalizer of some Sylow subgroup, therefore it is self-normalizing and intravariant in $X$. Since $U\cong D_{2(q-1)}$, its character theory is transparent. The group $U$ has exactly 2 linear characters and $(q-2)/2$ degree 2 real characters. Let $\eta_s\in \Irr(U)$ for $1\leq s \leq (q-2)/2$ such that $\eta_s(x^a)=\omega^{sa}+\omega^{-sa}$, where $\omega$ is a root of unity of order $q-1$. For $\chi=\Omega_s$ consider $\mu=\eta_s$. It is clear, by looking at the values at $x^a$ (notice that $\langle x \rangle$ is characteristic in $U$), that $(\Gamma\times \mathcal{G})_{\chi}=(\Gamma\times\mathcal{G})_{\mu}$. Let $\Sigma=X\rtimes\langle\alpha \rangle\cong \Aut(X)$. Call $\Delta=\mathbf{N}_{\Sigma}(U)\cong
\Gamma=\Gamma_{\chi^{\mathcal{G}}}$. To see the last equality, view the index $s$ of $\Omega_s$ as in $\mathbb{Z}/(q-1)\mathbb{Z}$ and observe that if $\delta\in\Gamma$, then $x^{\delta}$ is a generator of $H$, so it is $x^k$ for some $k$ coprime to $q-1$, hence $\Omega_s^{\delta^{-1}}=\Omega_{ks}$ from here it is clear that $\chi^\delta=\chi^\sigma$ for $\sigma\in\mathcal{G}$ sending $\omega\to\omega^k$. Finally observe that $|\Delta_{\mu}/U|$ divides $2n+1$ and $Z(X)=1$, therefore since $\chi,\mu$ are real-valued, $(\chi,\mu)$ satisfy the inductive Feit condition by Theorem \ref{crit}(d). \\

$(c)$ We now consider the real characters  $\Theta_l$ of degree $(q-1)(q-r+1)$ for $1\leq l \leq (q+r)/4$.  Let $H=\langle y \rangle$, we have that $H$ is a cyclic Hall subgroup of order $q+r+1$. Let $U=\mathbf{N}_X(H)$, then $|U:H|=4$ and $U=\mathbf{N}_X(L)$ for every nontrivial subgroup $L$ of $H$ (see section 1 of \cite{Bur}). As above, $U$ is self-normalizing and intravariant in $X$. Again, let $\Gamma=\Aut(X)_U$. Let $\zeta$ be a primitive root of unity in $\mathbb{C}$ of order $(q+1+r)$, then, the function $\mu_l$ defined by $\mu_l(y^b)=\zeta^{lb}+\zeta^{-lb}+\zeta^{lbq}+\zeta^{-lbq}$, and $\mu_l(u)=0$ for every $u\in U-H$ defines an irreducible character of $U$ (see Section 17 of \cite{Suz}). Now, for $\chi=\Omega_l$ we consider $\mu=\mu_l$, since for $y^b\neq 1$ we have that $\chi(y^b)=-\mu(y^b)$ we can verify directly (noticing that $\langle y \rangle$ is characteristic in $U$, so $\mu_l((y^b)^\delta)=-\theta((y^b)^{\delta})$ for $1\leq b\leq (q+r)/4$) that   $(\Gamma\times \mathcal{G})_{\chi}=(\Gamma\times\mathcal{G})_{\mu}$ and arguing as in the paragraph above, we have that $(\chi,\mu)$ satisfies the inductive Feit condition by Theorem \ref{crit}(d). Finally, the treatment of the characters $\Lambda_l$ of degree $(q-1)(q+r+1)$ is analogous to what we have just done but taking $U$ to be the normalizer of $\langle z \rangle$.
\end{proof}

Note that in the proof we have only used the fact that $q\neq 8$ when we obtained the universal cover, however, everything that we did after still holds for $X=\Sz(8)$. Next, we deal with this group. Full details of the following proof are given elsewhere and are available on request.

\begin{teo}\label{sz8}
    The simple group $S=\Sz(8)$ satisfies the inductive Feit condition.
\end{teo}
\begin{proof}
In this case the Schur multiplier of $S$ is the Klein Group $C_2 \times C_2$ (see Theorem 5.1.4 of \cite{KL}). Therefore $X=2^2.\Sz(8)$. Of course, $\mathbf{Z}(X)=C_2\times C_2$. The only nontrivial normal subgroups of $X$ are those generated by its central involutions. We have that $\Aut(X)\cong \Aut(S)$ and $|\Out(X)|=3$, moreover, table 6.3.1 of \cite{GLS} reveals that any outer automorphism permutes the central involutions of $X$.\\
Let $\chi\in \Irr(X)$, $\ker(\chi)$ is either $\mathbf{Z}(X)$ or $C_2\subseteq \mathbf{Z}(X)$. If $\ker(\chi)=\mathbf{Z}(X)$, so that $\chi\in \Irr(\Bar{X})=\Irr(\Sz(8))$. We can recycle work from the proof of Theorem \ref{suzuki trivial schur} to conclude that there exists $U<X$ self-normalizing and intravariant subgroup of $X$ and $\mu\in\Irr(U)$ such that $(\chi,\mu)$ satisfies the inductive Feit condition. If $\chi\in\Irr(X)$ with kernel of order 2, using computations with \textsc{GAP} (which provides the group $X$) we find for each such $\chi$ a self-normalizing and intravariant subgroup $U<X$ (which can always be taken to be the normalizer of a Sylow $p$-subgroup, $p\in\{2,7,13\}$) and an irreducible character $\mu$ of $U$ such that $(\chi,\mu)$ satisfies the inductive Feit condition using Theorem \ref{crit}(b).
\end{proof}

This concludes the proof of Theorem B.

\section{Inductive Galois-McKay condition}

Consider any simple group of the form $S=\PSL(2, p^n)$, recall that $|S|=p^n(p^n-1)(p^n+1)$ for $p$ even, and  $|S|=p^n(p^n-1)(p^n+1)/2$ for $p$ odd. From \cite{J}, \cite{R}, \cite{RS} it follows that the inductive Galois-McKay condition holds for $\ell=2$ and $\ell=p$. Thus, to completely verify that the condition holds for every prime dividing $|S|$, it remains (in the literature) to consider odd primes (say  $\ell$) dividing $q+1$ and  $q-1$. Our first goal is to prove this for $p$ odd. Then we will do it for $p=2$ and finally we will check the alternating group $A_6\cong\PSL_2(9)$.

\begin{teo}\label{teo3}
Let $p>2$ be any prime, $q=p^n$, $q\neq 9$. Then the simple group $S=\PSL_2(q)$ satisfies the inductive Galois-McKay condition for any odd prime $\ell$ dividing $q-1$.
\end{teo}
\begin{proof}
Let $\ell$ be an arbitrary but fixed odd prime dividing $q-1$, note that since there are no odd primes dividing $4$, we are assuming that $q\neq5$ . The universal covering group of $S$ is $X=\SL_2(q)$, we work with notation from setup 1 so recall $\Aut(X)=\langle S,\tau, \alpha \rangle$ where $\langle S,\tau\rangle \cong \PGL_2(q) $ and $\alpha$ acts by taking the $p$-th power to the matrix entries. We will use notation from Theorem 38.1 of \cite{D}. Consider the diagonal torus (split case) $T=<a>\cong C_{q-1}$, we take $L\in \Syl_{\ell}(X), L\subseteq T$. Let $U=\mathbf{N}_X(T)=\langle a,u |$ $a^{q-1}=1,$ $u^2=a^{(q-1)/2}$, $a^u=a^{-1}\rangle$, as mentioned before, $U$ is dicyclic of degree $(q-1)/2$. Also, $U=\mathbf{N}_X(L)$ and $U$ is self-normalizing and intravariant in $X$. Moreover, $\Aut(X)_U=\langle \bar{U}, \tau, \alpha \rangle $ where $\bar{U}=U/\mathbf{Z}(X)$ and  $\Gamma := \Aut(X)_L=\langle \bar{U}, \tau, \alpha \rangle = \Aut(X)_U$. \\

Note that $\Irr_{\ell'}(X)=\{1_X,\Psi,\xi_1,\xi_2,\chi_i:$ $1\leq i\leq (q-3)/2\}$. It is also convenient to keep in mind that if $q\equiv 1 \mod 4$, then $u,au\in (a^{(q-1)/4})$ and if $q\equiv 3 \mod 4$ then $u,au\in (b^{(q+1)/4})$. Also, note that $\Irr(U)=\Irr_{\ell'}(U)$. Recall the notation introduced above for the characters of $U$ (our $a$ here corresponds to $x$ there). Following Section 15 of \cite{IMN} we give the next bijections. If $4|(q-1)$ let

\begin{center}
    $\Omega:$ $\Irr_{\ell'}(X)\longrightarrow \Irr_{\ell'}(U)$, $
    1_X \longmapsto 1_U$, $\Psi \longmapsto \Sgn_U$, $\xi_1 \longmapsto \mu^+$, $\xi_2\longmapsto \mu^-$, $\chi_i\longmapsto\mu_i.$
\end{center}
\underline{Claim 1:} $\Omega$ is $\Gamma \times \mathcal{H}$-equivariant.

To begin, we note that $1_X, \Psi$ and $1_U, \Sgn_U$ are $\Gamma\times \mathcal{H}$-stable. Indeed, $\Psi$ is the unique irreducible character of $X$ of degree $q$ and for $\Sgn$, if $(\delta,\sigma)\in\Gamma\times\mathcal{H}$, by rationality  $\Sgn^{(\delta,\sigma)}=\Sgn_U^{\delta}=\Sgn_U$. The last equality follows for example from the fact that $a^{\delta^{-1}}\in (a^l)$ (note that $\tau$ fixes $a$ and $\bar{U}$,$\alpha$ of course send $a$ to a power of itself) for some integer $l$, so $\Sgn_U^{\delta}(a)=1$. Next we prove that $\Omega$ restricted to $\{\xi_1,\xi_2\}$ is $\Gamma\times \mathcal{H}$-equivariant. To see this we start by noticing that for $\delta\in \Gamma$, then $\delta$ either fixes $\xi_i$ or interchanges them, similarly it fixes $\mu^+,\mu^-$ or in interchanges them. The same happens for $\sigma\in \mathcal{H}$. Therefore it is enough to prove that $\Gamma_{\xi}=\Gamma_{\mu}$, and $\mathcal{H}_\xi=\mathcal{H}_{\mu}$ where $\xi$ denotes any but fixed element of $\{\xi_1,\xi_2\}$ and $\mu $ denotes any but fixed element of $\{\mu^+,\mu^-\}$. The first equality was proved in Theorem 15.3 of \cite{IMN}. We now argue that $\mathcal{H}_\xi=\mathcal{H}_{\mu}$. Assume first  that $n$ is even, then the characters $\mu,\xi$ are rational so fixed by $\mathcal{H}$. Assume now that $n$ is odd. It suffices to show that $\sigma\in \mathcal{H}$ fixes $\sqrt{p}$ when $q\equiv 1 \mod 4$ and that $\sigma(i)=-i$ if and only if $\sigma(i\sqrt{p})=-i\sqrt{p}$ when $q\equiv 3 \mod 4$. In our situation, since $\ell|q-1$ and $p^n=p^{2m+1}$, we have that $\left(\dfrac{p}{\ell}\right)=1$. So, if $p\equiv 1\mod 4$ then $q \equiv 1 \mod 4$ and by Quadratic reciprocity $$\left(\dfrac{\ell}{p}\right) \left(\dfrac{p}{\ell}\right)=(-1)^{\frac{p-1}{2}\frac{\ell-1}{2}}=1.$$ Therefore we have that $\left(\dfrac{\ell^m}{p}\right)=1 $ for any positive integer $m$. So, for any $\sigma\in \mathcal{H}$, since there exists a non-negative integer $f$ such that $\sigma(e^{\frac{2\pi i}{pi}})=(e^{\frac{2\pi i}{pi}})^{\ell^f}$ and $\sigma(i)=i^{\ell^f}$, we have by the Gauss Sum formula that $\sigma(\sqrt{p})=\sqrt{p}$. Finally, if $p\equiv 3\mod 4$, we have that (recycling notation from above) $\sigma(i\sqrt{p})=-i\sqrt{p}$ if and only if $\left(\dfrac{\ell^f}{p}\right)=-1$ which is equivalent to $f$ odd and $\ell\equiv 3 \mod 4$. It is straightforward to check that $\sigma(i)=i^{\ell^f}=-i$ if and only if $f$ is odd and $\ell\equiv 3 \mod 4$, concluding this part of the proof.

To complete the proof of Claim 1 we just need to show that the restriction of $\Omega$ to $\{\chi_j|$ $1\leq j\leq (q-3)/2\}$ is $\Gamma\times\mathcal{H}$-equivariant. Let $(\delta,\sigma)\in \Gamma\times \mathcal{G}$ where $\delta\in \alpha^{-s}\langle \bar
U, \tau\rangle$ with $s\in \mathbb{Z}_{\geq0}$. Note that $\langle \bar
U, \tau\rangle$ fixes both $\chi_j,\mu_j$
, also $\chi_j(a)=\mu_j(a)$, and the subindex $j$ of $\chi_j$ (resp. $\mu_j$) is determined by the value at $a$. Note that $\chi_j(a^{p^m})=\mu_j(a^{p^m})$ for $m\geq 0$ since $q-1\nmid 2p^m$. Thus, we have that $\chi_j^{(\delta,\sigma)}(a)=\sigma(\chi_j(a^{p^s}))=\sigma(\mu_j(a^{p^s}))=\mu_j^{(\delta,\sigma)}(a)$, it follows that $\Omega(\chi_j^{(\delta,\sigma)})=\mu_j^{(\delta,\sigma)}$ proving Claim 1. \\

To verify the inductive Galois-McKay condition for $\ell$ it remains to prove the following.\\
\underline{Claim 2:} For every $\chi\in\Irr_{\ell'}(X)$ we have that
$$
(X\rtimes \Gamma_{\chi^{\mathcal{H}}}, X, \chi)_{\mathcal{H}} \succeq_{\mathbf{c}}(U\rtimes\Gamma_{\chi^{\mathcal{H}}}, U, \Omega(\chi))_{\mathcal{H}}.
$$

We will use the letter $\mu$ to denote $\Omega(\chi)$ for a fixed $\chi$. Observe that for every $\chi\in\Irr_{\ell'}(X)$ by Claim 1 we have that $(\Gamma \times \mathcal{H})_{\chi}=(\Gamma \times \mathcal{H})_{\mu}$. Let us first consider the character $\chi=1_X$, in this case it is clear that our conclusion holds by Theorem \ref{crit}(a). Let us now consider the Steinberg character $\chi=\Psi$, if $4|(q-1)$ we can directly compute $[\Psi_U,\Sgn_U]=1$ so
 we are again done by Theorem \ref{crit}(a). If $q\equiv 3 \mod 4$ we have that $[\Psi_U,\Sgn_U]=2$ we need to work a bit more. Let $M=\{y\in \GL_2(q)$ : $\det(y)=\pm 1\}$ as in the Setup 1. Then $U^{\sharp}:=\mathbf{N}_{M}(U)=U\rtimes C_2$ and we have that $\tau$ is induced by some element in $U^\sharp$ when acting by conjugation in $X$. Note that $M/X\cong U^{\sharp}/U\cong C_2$. Since $\mu=\Sgn_U$ is clearly $U^\sharp$-invariant (unique nontrivial linear rational), and we have an split extension, $\mu$ extends to a rational valued character $\mu^\sharp\in \Irr(U^{\sharp})$.
By Gallagher's Theorem (6.17 of \cite{Is}, whenever we mention Gallagher we mean this result), there are precisely two extensions of $\mu^\sharp$, $\nu\mu^\sharp$ of $\mu$ to $U^{\sharp}$, where $\nu$ denotes the nontrivial character of order 2 of $M/X\cong U^{\sharp}/U$.
Define $\Sigma=M\rtimes \langle\alpha\rangle\leq A$
, where $A$ is as in setup 1. The group $\Sigma=\Sigma_{\chi}$ contains $X$ as a normal subgroup and induces the full $\Aut(X)$ when acting on $X$. Also, $\Delta=\mathbf{N}_{\Sigma}(U)=U^{\sharp}\rtimes \langle\alpha\rangle$ induces $\Gamma_{\chi^{\mathcal{H}}}=\Gamma$. Note that $|\Sigma:M|=n=|\Delta:U^{\sharp}|$ is odd. We have that $\chi=\Psi$  extends to a rational valued character $\chi^*\in \Irr(\Sigma)$ (see Theorem A of \cite{MS}), and since $\alpha $ acts trivially on the $C_2$ part we see that $\mu^\sharp$ is $\Delta$-invariant hence $\mu^\sharp$ extends to a rational valued $\mu^*\in \Irr(\Delta)$. Clearly both extensions are $(\Delta\times \mathcal{H})_{\chi}$-invariant therefore since $\mathbf{C}_{\Delta}(X)=\{\pm1\}=\langle z\rangle$ and both characters $ \chi^*,\mu^*$ restrict to multiples of the trivial character of $\langle z \rangle$, we can conclude by Theorem \ref{crit}(c) that $(X\rtimes \Gamma_{\chi^{\mathcal{H}}}, X, \chi)_{\mathcal{H}} \succeq_{\mathbf{c}}(U\rtimes\Gamma_{\chi^{\mathcal{H}}}, U, \Omega(\chi))_{\mathcal{H}}.$\\

Now consider $\chi\in \{\xi_1,\xi_2\}$. One can easily check (doing 2 cases depending on $q\mod4$) that $[ \chi_U, \mu]=1$ so we are done by Theorem \ref{crit}(a). Finally let $\chi=\chi_j$ for some $1\leq j\leq (q-3)/2$. The proof of $\chi_j$ satisfying the inductive Feit condition implies (since the local subgroup $U$ that is used is the normalizer of unique up to conjugacy $C_{q-1}$ subgroup of $X$ and the considered $\mu$ is the adequate one) that $(X\rtimes \Gamma_{\chi_j^{\mathcal{G}}}, X, \chi_j)_{\mathcal{G}} \succeq_{\mathbf{c}}(U\rtimes\Gamma_{{\chi_j}^{\mathcal{G}}}, U, \mu_j)_{\mathcal{G}}.$ But since we have that $\Gamma=\Gamma_{{\chi_j}^{\mathcal{G}}}$ , then $\Gamma_{{\chi_j}^{\mathcal{H}}}\leq\Gamma=\Gamma_{{\chi_j}^{\mathcal{G}}}$ so we have
$(X\rtimes \Gamma_{\chi_j^{\mathcal{H}}}, X, \chi_j)_{\mathcal{H}} \succeq_{\mathbf{c}}(U\rtimes\Gamma_{{\chi_j}^{\mathcal{H}}}, U, \mu_j)_{\mathcal{H}}$ by Definition \ref{central order}. This concludes the proof of the theorem.
\end{proof}

We now deal with the remaining primes. We remark that the hardest part of the next proof is to prove the central ordering for the characters of degree $(q-1)/2$. The rest is kind of similar to the preceding proof. The problem here with the characters $\eta$ of degree $(q-1)/2$ is that after building the bijection $\Omega$, we have that $[\eta,\Omega(\eta)]=0$, so we cannot apply Theorem \ref{crit}(a). We will end up applying part (c) of that theorem. To do so, we will need to find preferred extensions of these characters. Since these characters are Weil characters, a natural thing to do is to embed our group in a larger dimensional symplectic group. Let us recall some basic facts about the representation theory of $\Sp_{2n}(p)$. We refer the reader to Section 2 and Section 4 of \cite{KT} for a more detailed discussion of what we will mention now. \\

Fix $p$ an odd prime and $n\geq1$ an integer (if $p=3$, assume $n>1$). The nontrivial irreducible characters of $\Sp_{2n}(p)$ of smallest degree are the so-called Weil characters; there are exactly four of them (see Theorem 5.2 of \cite{TZ}), two of smaller degree $(p^n-1)/2$, which we will denote by $\Phi_1,\Phi_2$ and two of larger degree $(p^n+1)/2$ which we will denote by $\Psi_1,\Psi_2$ (not to confuse with the Steinberg character of $\SL_2(q)$, which we have denoted by $\Psi$). The field of values of all these characters is $\mathbb{Q}(\sqrt{\epsilon_pp})$ where $\epsilon_p=(-1)^{\frac{p-1}{2}}$ (see Lemma 13.5 of \cite{G}). The following is a key fact for us. Let $\Sigma=\SL_2(q)\rtimes\langle\alpha\rangle$ as in setup 1, recall $\alpha$ is a field automorphism of order $n$ acting by taking power $p$ to the entries. Fix $\chi=\eta_i$ one of the two characters of $\SL_2(q)$ of degree $(q-1)/2$. Then there exists an embedding $$
\Sigma \xhookrightarrow{} \Sp_{2n}(p)
$$
in a way that $\chi$ is a restriction of a Weil character of $\Sp_{2n}(p)$. This is essentially Lemma 4.3 of \cite{KT}, and we recommend the reader to take into account the warning in \cite{KT} before this lemma.
\begin{teo}\label{GM odd}
Let $p>2$ be any prime, $q=p^n$, $q\neq 9$. Then the simple group $S=\PSL_2(q)$ satisfies the inductive Galois-McKay condition for any odd prime $\ell$ dividing $q+1$.
\end{teo}
\begin{proof}
Let $\ell$ be an arbitrary but fixed odd prime dividing $q+1$. The universal covering group of $S$ is $X=\SL_2(q)$, we view it as $X=\SU_2(q)$ and work with notation from setup 2.  We have that $S\cong \PSU_2(q)$ and also $\Aut(X)=
\langle S,\tau, \alpha \rangle$ where $\langle S,\tau\rangle \cong \PGU_2(q) $ and $\alpha$ acts by taking the $p$-th power to the matrix entries after fixing an orthonormal basis. The automorphism $\alpha$ has order $2n$, also in $\Out(X)$ has order $n$ (recall that raised to power $n$ acts as conjugation by $u\in \SU_2(q)$). \\ For the character table we will use notation from Theorem 38.1 of \cite{D}. Consider the non-split torus $T=<b>\cong C_{q+1}$, which is the unique up to conjugacy cyclic subgroup of order $q+1$
. We take $L\in \Syl_{\ell}(X), L\subseteq T$. Let $U=\mathbf{N}_X(T)=\langle b,u |$ $b^{q+1}=1,$ $u^2=b^{(q+1)/2}$, $b^u=b^{-1}\rangle$, recall $U$ is dicyclic of degree $(q+1)/2$ (so now, when $4|(q-1)$ we will have complex valued linear characters in $U$). Also recall that, $U=\mathbf{N}_X(L)$ so it is self-normalizing and intravariant in $X$. Finally, recall that $\Gamma := \Aut(X)_L=\langle \bar{U}, \tau, \alpha \rangle = \Aut(X)_U$  where $\bar{U}=U/\mathbf{Z}(X)$.\\

Note that $\Irr_{\ell'}(X)=\{1_X,\Psi,\eta_1,\eta_2,\theta_i:$ $1\leq i\leq (q-1)/2\}$. It is also convenient to keep in mind that if $q\equiv 1 \mod 4$, then $u,bu\in (a^{(q-1)/4})$ and if $q\equiv 3 \mod 4$ then $u,bu\in (b^{(q+1)/4})$. Also, note that $\Irr(U)=\Irr_{\ell'}(U)$. Recall the notation introduced above for the characters of $U$ (our $b$ here corresponds to $x$ there). We now give the next bijection. 

\begin{center}
    $\Omega:$ $\Irr_{\ell'}(X)\longrightarrow \Irr_{\ell'}(U)$, $
    1_X \longmapsto 1_U$, $\Psi \longmapsto \Sgn_U$, $\eta_1 \longmapsto \mu^+$, $\eta_2\longmapsto \mu^-$, $\theta_i\longmapsto\mu_i.$
\end{center}
\underline{Claim 1:} $\Omega$ is $\Gamma \times \mathcal{H}$-equivariant.

To begin, we note that $1_X, \Psi$ and $1_U, \Sgn_U$ are $\Gamma\times \mathcal{H}$-stable. Indeed, $\Psi$ is the unique irreducible character of $X$ of degree $q$ and for $\Sgn_U$, if $(\delta,\sigma)\in\Gamma\times\mathcal{H}$, by rationality  $\Sgn^{(\delta,\sigma)}=\Sgn_U^{\delta}=\Sgn_U$. The last equality follows for example since $b^{\delta^{-1}}\in (b^m)$ for some integer $m$, so $\Sgn_U^{\delta}(b)=1$. Next we prove that $\Omega$ restricted to $\{\eta_1,\eta_2\}$ is $\Gamma\times \mathcal{H}$-equivariant. To see this we start by noticing that for $\delta\in \Gamma$, then $\delta$ either fixes $\eta_i$ or interchanges them, similarly it fixes $\mu^+,\mu^-$ or it interchanges them. The same happens for $\sigma\in \mathcal{H}$. Therefore it is enough to prove that $\Gamma_{\eta}=\Gamma_{\mu}$, and $\mathcal{H}_\eta=\mathcal{H}_{\mu}$ where $\eta$ denotes any but fixed element of $\{\eta_1,\eta_2\}$ and $\mu$ denotes any but fixed element of $\{\mu^+,\mu^-\}$. The first equality follows from Theorem 15.3 of \cite{IMN}. We argue as in the previous proof that $\mathcal{H}_\eta=\mathcal{H}_{\mu}$. \\
If $n$ is even then $\eta$ is rational so fixed by $\mathcal{H}$. Now $\mu$ is not rational, however $\sigma\in\mathcal{H}$ fixes it if and only if $\sigma(i)=i$. It is enough to show that $4|(\ell-1)$. To see this write $n=2m$ and use the fact that $p^{2m}\equiv -1 \mod \ell$. From this $(-1)^{(\ell-1)/2}=\left(\dfrac{-1}{\ell}\right)=1$. \\
Assume now that $n$ is odd. Fix $\sigma\in\mathcal{H}$, consider a non-negative integer $f$ such that $\sigma(e^{\frac{2\pi i}{pi}})=(e^{\frac{2\pi i}{pi}})^{\ell^f}$ and $\sigma(i)=i^{\ell^f}$
If $q\equiv 1 \mod 4$, we have to show that $\sigma(\sqrt{p})=\sqrt{p}$ if and only if $\sigma(i)=i$. If $q\equiv 3 \mod 4$, we have to show that $\sigma(i\sqrt{p})=i\sqrt{p}$. We need to prove the same for both cases so we will only worry about $p\mod 4$. Assume first that $4|(p-1)$, from the Gauss Sum formula we get that $\sigma(\sqrt{p})=\sqrt{p}$ if and only if $\left(\dfrac{\ell^f}{p}\right)=1$. Our hypothesis easily yield $\left(\dfrac{p}{\ell}\right)=(-1)^{\frac{\ell-1}{2}}$, so using that $(p-3)/2$ is odd we get $\left(\dfrac{\ell}{p}\right)=(-1)^{\frac{\ell-1}{2}}$ via Quadratic reciprocity. Therefore, $\sigma(\sqrt{p})=\sqrt{p}$ if and only if $4|(\ell^f-1)$, which is equivalent to $\sigma(i)=i$ as wanted. If $4|(p-3)$ the same computations show via the Gauss Sum formula that $\sigma(i\sqrt{p})=i\sqrt{p}$.
To complete the proof of Claim 1 we just need to show that the restriction of $\Omega$ to $\{\theta_j|$ $1\leq j\leq (q-1)/2\}$ is $\Gamma\times\mathcal{H}$-equivariant. Let $(\delta,\sigma)\in \Gamma\times \mathcal{G}$ where $\delta\in \alpha^{-s}\langle \bar
U, \tau\rangle$ with $s\in \mathbb{Z}_{\geq0}$. Note that $\langle \bar
U, \tau\rangle$ fixes both $\theta_j,\mu_j$ 
, also $\theta_j(b)=-\mu_j(b)$, and the subindex $j$ of $\theta_j$ (resp. $\mu_j$) is determined by the value at $b$. Note that $\theta_j(b^{p^m})=-\mu_j(b^{p^m})$ for $m\geq 0$ since $q+1\nmid 2p^m$. Thus, we have that $\theta_j^{(\delta,\sigma)}(b)=\sigma(\chi_j(b^{p^s}))=\sigma(-\mu_j(b^{p^s}))=-\mu_j^{(\delta,\sigma)}(b)$, it follows that $\Omega(\theta_j^{(\delta,\sigma)})=\mu_j^{(\delta,\sigma)}$ proving Claim 1. \\

To verify the inductive Galois-McKay condition for $\ell$ it remains to prove the following.\\
\underline{Claim 2:} For every $\chi\in\Irr_{\ell'}(X)$ we have that
$$
(X\rtimes \Gamma_{\chi^{\mathcal{H}}}, X, \chi)_{\mathcal{H}} \succeq_{\mathbf{c}}(U\rtimes\Gamma_{\chi^{\mathcal{H}}}, U, \Omega(\chi))_{\mathcal{H}}.
$$

 We will use the letter $\mu$ to denote $\Omega(\chi)$ for a fixed $\chi$. Observe that for every $\chi\in\Irr_{\ell'}(X)$ by Claim 1 we have that $(\Gamma \times \mathcal{H})_{\chi}=(\Gamma \times \mathcal{H})_{\mu}$. Assume that $\chi=1_X$, then clearly $[\chi,\Omega(\chi)]=1$, hence we are done by Theorem \ref{crit}(a). Now take $\chi=\Psi$, by direct computation if $4|(q-3)$ we have that $[\Psi,\Omega(\Psi)]=1$ so we are done again. However, if $4|(q-1)$ we have that $[\Psi,\Omega(\Psi)]=0$. Let $M=\{y\in\GU_2(q)$ : $\det(y)=\pm 1\}$ as in setup 2.  Then $U^\sharp:=\mathbf{N}_{M}(U)=U\rtimes C_2$, and $\tau$ is induced by some element in $U^\sharp$ when acting by conjugation in $X$. Note that $\mu=\Sgn_U$ is $U^{\sharp}$-invariant, and since the extension is split, $\mu$  extends naturally to a rational valued character $\mu^\sharp \in U^\sharp$ which takes value 1 in $C_2\leq U^\sharp$. Now consider  the group $\Sigma=M\rtimes \langle \alpha \rangle \leq A$, in the notation introduced in setup 2. The group $\Sigma$ induces the full $\Aut(X)$ while acting on $X$. Then $\Delta=\mathbf{N}_{\Sigma}(U)=U^\sharp\rtimes  \langle \alpha \rangle $ induces $\Gamma=\Gamma_{\chi}$. Note that by Theorem A of \cite{MS}, the Steinberg character $\Psi$ extends to a rational valued character, say, $\chi^*\in \Irr(\Sigma)$ which is clearly $(\Delta\times \mathcal{H})-$invariant hence we need to find a $(\Delta\times \mathcal{H})-$invariant extension of $\mu$ to a  character $\mu^*\in\Irr(\Delta)$ such that$[\chi^*_{\mathbf{C}_{\Delta}(X)},\mu^*_{\mathbf{C}_{\Delta}(X)}]\neq 0$. If so
 we will conclude that $(X\rtimes \Gamma_{\chi^{\mathcal{H}}}, X, \chi)_{\mathcal{H}} \succeq_{\mathbf{c}}(U\rtimes\Gamma_{\chi^{\mathcal{H}}}, U, \Omega(\chi))_{\mathcal{H}}$ by Theorem \ref{crit}(c).
 Note that since $\alpha$ acts trivially on the $C_2$ part of $U^\sharp$, $\mu^\sharp$ is $\Delta$-invariant. Since the extension to $\Delta$ is split, we have control over the extensions of $\mu^\sharp$ to $\Delta$. Now note (see setup 2) that $\mathbf{C}_{\Delta}(X)=\langle u^{-1}\alpha^n\rangle\cong C_4$. It contains $\mathbf{Z}(X)$ and $(\chi^*)_{\mathbf{Z}(X)}=q1_{\mathbf{Z}(X)}$ hence the only possible irreducible constituents of $(\chi^*)_{\mathbf{C}_{\Delta}(X)}$ are the two rational characters $1_{C_4}$ and $\rho$ of $C_4$. If $[(\chi^*)_{C_4},\rho]\neq 0$ we consider by Gallagher's correspondence $\mu^*\in \Irr(\Delta)$ such that $\mu^*(\alpha)=1$. In this case, since $\mu(u)=-1$, we have that $(\mu^*)_{C_4}=\rho\in\Irr(C_4)$ so $[\chi^*_{\mathbf{C}_{\Delta}(X)},\mu^*_{\mathbf{C}_{\Delta}(X)}]\neq 0$. Note that $\mu^*$ is clearly $(\Delta\times \mathcal{H})-$invariant so we are done in this case. On the other hand assume that $[(\chi^*)_{C_4},\rho]= 0$, it is clear that $[(\chi^*)_{C_4},1_{C_4}]\neq 0$. It is enough to find a $\mathcal{H}$-invariant extension of $\mu^\sharp$ to $\Delta$ such that $\mu^*(\alpha^n)=-1$ (if so $(\mu^*)_{C_4}=1_{C_4}$). To find this, it is enough to apply Gallagher to obtain an extension $\mu^*\in\Irr(\Delta)$ such that if $n=ms$ with $s=2^k$ and $2\nmid m$, $\mu^*(\alpha)=\xi_{2s}$ where $\xi_{2s}$ is a root of unity of order $2s$. In this case, $\mu^*(\alpha^n)=-1$ and we just need to show that the extension is invariant under $\sigma\in\mathcal{H}$. To see this note that since $\ell$ is odd, $\ell|p^{sm}+1$ implies $\ell\equiv 1$mod$(2s)$, thus $\sigma(\xi_{2s})=\xi_{2s}$ as wanted.\\

Now, let $\chi=\theta_j$ for some $1\leq j\leq (q-1)/2$. By the proof of  Theorem 8.4 \cite{BKNT}, in which it is proved that  $\theta_j$ satisfies the inductive Feit condition, since it uses the normalizer of the unique (up to conjugation) $C_{q+1}$ subgroup of $X$ and the precise $\mu$, it follows that for this $U$ we have that $(X\rtimes \Gamma_{\theta_j^{\mathcal{G}}}, X, \theta_j)_{\mathcal{G}} \succeq_{\mathbf{c}}(U\rtimes\Gamma_{{\theta_j}^{\mathcal{G}}}, U, \mu_j)_{\mathcal{G}}.$ Note that $\bar{U},\tau$ fix $\theta_j$ and it is easy to see by acting on $b$ that $\theta_j^{\alpha }=\theta_j^{\sigma}$ for some $\sigma\in \mathcal{G}$ hence $\Gamma=\Gamma_{{\theta_j}^{\mathcal{G}}}$ . So $\Gamma_{{\theta_j}^{\mathcal{H}}}\leq\Gamma=\Gamma_{{\theta_j}^{\mathcal{G}}}$ so we have that 
$(X\rtimes \Gamma_{\theta_j^{\mathcal{H}}}, X, \theta_j)_{\mathcal{H}} \succeq_{\mathbf{c}}(U\rtimes\Gamma_{{\theta_j}^{\mathcal{H}}}, U, \mu_j)_{\mathcal{H}}$ by Definition \ref{central order}. \\

Finally, consider $\eta=\chi\in \{\eta_1,\eta_2\}$. One can easily check (doing 2 cases depending on $q\mod4$) that $[ \chi_U, \mu]=0$. We start by assuming that $n$ is odd, $q\equiv 1 \mod 4$ and $\eta_i^{\mathcal{H}}=\{\eta_1,\eta_2\}$. Let, as in setup 2,  $M=\{ y\in \GU_2(q)$ : $\det(y)=\pm1\}$, which induces $\tau$ when acting on $X$  and consider $\Sigma=M\rtimes \langle\alpha^2\rangle\leq A$. Since $n$ is odd, it induces the full $\Aut(X)$ while acting on $X$ by conjugation. Note in this case that $\mathbf{C}_{\Sigma}(X)=\mathbf{Z}(X)=\langle z \rangle$. Then, $\Delta=\mathbf{N}_{\Sigma}(U)$ induces $\Gamma=\Gamma_{\chi_i^{\mathcal{H}}}$, where $\chi_i=\eta_i$. We have that $\Sigma_{\chi_i}=X\rtimes \langle\alpha^2\rangle$ and $\Delta_{\mu_i}= \langle U,\alpha^2\rangle$, also note that $\Delta_{\mu}/U\cong C_n$. We let $\chi_i^*$ be the unique real extension to $\Sigma_{\chi_i}$, using that $n$ is odd. If $(\delta,\sigma)\in (\Delta\times\mathcal{H})_{\chi_i}$, since $\sigma$ preserves reality (descends to automorphism of cyclotomic field so commutes with complex conjugation), by uniqueness, it is clear that $(\chi^*)^{(\delta,\sigma)}=\chi^*$. In this case, for $\mu$, since  $\Delta_{\mu}/U= C_n$ and $n$ is odd, we appeal to Theorem A of \cite{Nav}, to find $\mu^*$ the unique $\mathbb{Q}(i)$-valued character of $\Delta_{\mu}$ extending $\mu$. Since $\sigma\in\mathcal{H}$ descends to an automorphism of $\mathbb{Q}(i)$, it is clear (by uniqueness) that $(\mu^*)^{(\delta,\sigma)}=\mu^*$ for any $(\delta,\sigma)\in (\Delta\times\mathcal{H})_{\chi}= (\Delta\times\mathcal{H})_{\mu}$. Since both $\chi^*$, $\mu^*$ restricted to $\mathbf{Z}(X)$ are multiples of the nontrivial irreducible character of $\mathbf{Z}(X)$,  we are  done by Theorem \ref{crit}(c). If we had $\eta_i^{\mathcal{H}}=\{\eta_i\}$ we would do the same argument but with $\Sigma=X\rtimes \langle \alpha^2 \rangle$, so that $\mathbf{N}_{\Sigma}(U)$ induces $\Gamma_{\chi^{\mathcal{H}}}$ and we can apply Theorem \ref{crit}(c).\\

It is convenient for us to go back to $\SL_2(q)$, so formally, fix an isomorphism $\theta:\SU_2(q)\to\SL_2(q)$ so that our $U\subseteq \SU_2(q)$ is mapped to the normalizer of a (unique up to conjugacy) $C_{q+1}$ subgroup of $\SL_2(q)$, call it $U$ again. 
Rename $X=\SL_2(q)$, it is enough to prove that $(X\rtimes \Gamma_{\eta^{\mathcal{H}}}, X, \eta)_{\mathcal{H}} \succeq_{\mathbf{c}}(U\rtimes\Gamma_{{\eta}^{\mathcal{H}}}, U, \mu)_{\mathcal{H}}$ where $\Gamma=\Aut(X)_U$, and $\mu$ denotes any linear character of $U$ different from $1_U,\Sgn_U$. Here, $\Aut(X)=\langle S,\tau,\alpha \rangle$ where $\langle S,\tau\rangle\cong \PGL_2(q)$, and $\alpha$ is induced by the action of the field automorphism $x\to x^p$ on $\SL(\mathbb{F}_q^2)$. The isomorphism $\theta$ induces an isomorphism $\Aut(\SU_2(q))\to \Aut(\SL_2(q))$ in which inner automorphisms are mapped to inner automorphisms. Call $\hat{\tau},\hat{\alpha}$ the images of $\tau,\alpha\in \Aut(\SU_2(q))$, the last sentence allows us to see that the order of $\hat{\tau}$,$\hat{\alpha}$ in $\Out(X)$ is $2,n$, respectively, moreover,  $\langle \hat{\alpha},\hat{\tau}\rangle$ viewed mod $\Inn(X)$ also generate $\Out(X)$. We know that  $U$ is self-normalizing in $X$, and we have that $|\Gamma|=2(q+1)n$ and $|\Out(X)|=2n$. Moreover, $\Gamma=\Aut(X)_U=\langle \bar{U},\hat{\alpha},\hat{\tau}\rangle$. Since $\hat{\alpha}$ fixes $\eta_i$, $\hat{\alpha}=x\alpha^t$ for some $x\in\Inn(X)$, $t\in \mathbb{Z}$, in fact, in $\Out(X)$, both $\hat{\alpha
}$ and $\alpha$ generate the same group. Also, $\hat{\tau}$ fuses $\eta_i$, so perhaps after multiplying by some power of $\hat{\alpha}$ we have that $\hat{\tau}=\tau$ in $\Out(X)$. The property that $(\Gamma\times \mathcal{H})_{\eta}=(\Gamma\times\mathcal{H})_{\mu}$ is still true so $\hat{\tau}$ fuses $\eta_1$ with $\eta_2$, seen as characters of $X=\SL_2(q)$ and $\mu^+$ with $\mu^-$, as characters of $U\subseteq\SL_2(q)$.\\

If $q\equiv 3 \mod 4$, note that $\mathbb{Q}(\chi)=\mathbb{Q}(i\sqrt{p})$. Let $\Sigma=X\rtimes\langle\alpha \rangle$, then $\Delta=\mathbf{N}_{\Sigma}(U)=\langle U,\hat{\alpha}\rangle$ and it induces $\Gamma_{\chi^{\mathcal{H}}}=\Gamma_{\chi}$ ( here and in what follows when we say $\hat{
\alpha}$ we mean $x\alpha^t\in\Sigma$, so now $x$ is viewed as an element of $X$). We now appeal to the fact that fixed $\chi=\eta_i$, $\Sigma$ embbeds into $\Sp_{2n}(p)$ in a way so that $\chi$ is the restriction of a Weil character of $\Sp_{2n}(p)$. Recall that $\Sp_{2n}(p)$ has two characters $\Phi_1,\Phi_2$ of degree $(q-1)/2$ and field of values $\mathbb{Q}(i\sqrt{p})$ (note that $p\equiv 3\mod 4)$. We may assume that $\Phi_1$ extends $\chi$, and we call $\chi^*$ the restriction of $\Phi_1$ to $\Sigma=\Sigma_{\chi}$. Note that by the proof of claim 1, if $\sigma\in \mathcal{H}$, $\sigma(i\sqrt{p})=i\sqrt{p}$. We have that for $(\delta,\sigma)\in (\Delta\times\mathcal{H})_{\chi}$, $(\chi^*)^{(\delta,\sigma)}=(\chi^*)^{\delta}=\chi^*$ since $\delta\in\Delta\subseteq\Sigma=\Sigma_{\chi}$. In the $U$ context, our character $\mu$ is rational so we may consider $\mu^*$ the unique rational extension to $\Delta_{\mu}=\Delta$ (since $\Delta_{\mu}/U=C_n$, and n is odd), and for any $(\delta,\sigma)\in (\Delta\times\mathcal{H})_{\chi}$ by uniqueness, it is clear that $(\mu^*)^{(\delta,\sigma)}=\mu^*$. So we are done by Theorem \ref{crit}(c) since both characters $\chi^*,\mu^*$ lie over the trivial character of $\mathbf{Z}(X)=\mathbf{C}_{\Delta}(X)$.\\

We now assume $n$ is even. The situation here is harder so we will go carefully, knowing that our ultimate goal is to apply Theorem \ref{crit}(c). Write $n=sm$ where $s=2^k$, $k\geq 1$, and $m$ is odd. Let $\Sigma=X\rtimes\langle\alpha\rangle$. Fix $\chi=\eta_i$, it extends to $\Sigma$. We try to understand preferred extensions of $\chi$ and the Galois action up on these extensions first, then we will go to the $U$ side. Again, denote by $\Phi_1,\Phi_2$ the two Weil characters of $\Sp_{2n}(p)$ of degree $(q-1)/2$ and by $\Psi_1,\Psi_2$ the two Weil characters of $\Sp_{2n}(p)$ of degree $(q+1)/2$, recall that all of them have field of values $\mathbb{Q}(\sqrt{\epsilon_pp})$, where $\epsilon_p=(-1)^{(p-1)/2}$. We may assume that $w=\Phi_1+\Psi_1$ is a total Weil character. Again, we have an embedding
$$
X\rtimes \langle\alpha^m\rangle \subseteq \Sigma \xhookrightarrow{} \Sp_{2n}(p)
$$
in a way so that our character $\chi$ extends to a Weil character, say $\Phi_1$, of $\Sp_{2n}(p)$.
Let $\chi^{\sharp}=\Phi_1|_{X\rtimes\langle\alpha^m\rangle}$, we claim that $\mathbb{Q}(\chi^{\sharp})=\mathbb{Q}(\sqrt{\epsilon_pp})$. To see this, since $|\mathbb{Q}(\sqrt{\epsilon_pp}):\mathbb{Q}|=2$ is enough to see that $\chi^{\sharp}$ is not rational. Note that $\Psi_1$ restricts to an irreducible character $\xi$ of degree $(q+1)/2$ of $X$. This $\xi$ is rational and has odd degree, and since there is only one linear character in $X$, $o(\xi)=1$. Hence, by the proof of Corollary 2.4 of \cite{NT} we have that there is a rational extension of $\xi$, say $\xi^{\sharp}\in \Irr(X\rtimes\langle\alpha^m\rangle)$. By Gallagher's theorem, there exists $\lambda\in\Irr(C_s)$ such that $\lambda\xi^{\sharp}=\Psi_1|_{X\rtimes\langle\alpha^m\rangle}$. From this we can see that $\mathbb{Q}(\Psi_1|_{X\rtimes\langle\alpha^m\rangle})=\mathbb{Q}$, indeed $\mathbb{Q}(\Psi_1|_{X\rtimes\langle\alpha^m\rangle})\subseteq \mathbb{Q}(\sqrt{\epsilon_pp})\subseteq \mathbb{Q}(e^{2\pi i/p})$, on the other hand $\mathbb{Q}(\lambda\xi^{\sharp})=\mathbb{Q}(\lambda)\subseteq \mathbb{Q}(e^{2\pi i/s})$ and $\mathbb{Q}(e^{2\pi i/s})\cap \mathbb{Q}(e^{2\pi i/p})=\mathbb{Q}$. Now, applying Theorem 3.5 of \cite{KT2} we find $h\in X\rtimes \langle\alpha^m\rangle$ such that $|w(h)|^2=p^m$. Hence $|\Psi_1(h)+\chi^{\sharp}(h)|\not\in \mathbb{Q}$ so $\chi^{\sharp}$ is not rational as wanted.

Let $\chi^*=\Phi_1|_{\Sigma}$, it follows that $\mathbb{Q}(\chi^*)=\mathbb{Q}(\sqrt{\epsilon_pp})$. Let $\Delta=\mathbf{N}_{\Sigma}(U)$, and let $(\delta,\sigma)\in (\Delta\times\mathcal{H})_{\chi}$, note that $(\chi^*)^{(\delta,\sigma)}=(\chi^*)^{\sigma}$. Now, $(\chi^*)^{\sigma}\neq \chi^*$ if and only if $\sigma$ is the nontrivial element of $\Gal(\mathbb{Q}(\sqrt{\epsilon_pp})/\mathbb{Q})$, and this, by the Gauss Sum formula happens if and only if $\left(\dfrac{\ell}{p}\right)=-1$ and $f$ is odd (here we are fixing a non-negative $f$ depending on $\sigma$, such that for primitive roots of unity of order $p$ and $4s$ say $\xi_p,\xi_{4s}$, $\sigma(\xi_p)=\xi_p^{\ell^f}$ and $\sigma(\xi_{4s})=\xi_{4s}^{\ell^f})$. If we have such $\sigma$, then by going up to $\Sp_{2n}(p)$ and applying $\sigma$ to $\Phi_1$ we see that $(\chi^*)^{\sigma}=\Phi_2|_{\Sigma}$. To fully understand the situation it remains to see what $\lambda\in \Irr(C_n)$ satisfies $\Phi_2|_{\Sigma}=\lambda\chi^*$. First we observe that $\mathbb{Q}(\lambda)\subseteq\mathbb{Q}(\sqrt{\epsilon_pp})$. From this we now prove that $\mathbb{Q}(\lambda)=\mathbb{Q}$ and hence it is the unique character of $C_n$ of order 2. To see this
start by noticing that since $\mathbb{Q}(\lambda)\subseteq\mathbb{Q}(\sqrt{\epsilon_pp})$, $\lambda|_{C_s}$ is rational and of course it is not trivial ($\mathbb{Q}(\chi^{\sharp})=\mathbb{Q}(\sqrt{\epsilon_pp})$ and $\sigma$ acts as the nontrivial element of $\Gal(\mathbb{Q}(\sqrt{\epsilon_pp})/\mathbb{Q})$, so moves $\chi^{\sharp}$). Now $\lambda\in \Irr(C_n)$ extends the order 2 character of $C_s$ and takes values in $\mathbb{Q}(\sqrt{\epsilon_pp})$. Since $\mathbb{Q}(\lambda)=Q(\xi_t)$ for some primitive $t$-th root of unity $\xi_t$, where $t|n$, we see directly (looking at degrees of field extensions and observing that $\mathbb{Q}(\xi_t)\neq\mathbb{Q}(i)$) that $\lambda$ has to be rational as long as we are not in the case $p=3$ and $3|n$. In the case $p=3$, $3|n$ we have to argue that we can not be in the situation of having our $\lambda\in \Irr(C_n)$ with $\mathbb{Q}(\lambda)=\mathbb{Q}(i\sqrt{3})$. These characters send a generator of $C_n$ to either a primitive $3$rd root of unity, say $\omega$, or to a primitive $6$-th root of unity, say $-\omega$. The characters of order 3  never take the value $-1$ and therefore can not extend $\lambda|_{C_s}$. It only remains to assume that $\Phi_2|_{\Sigma}=\lambda\Phi_1|_{\Sigma}$ with $o(\lambda)=6$. On one hand, $\det(\Phi_i)=1$, (go up to $\Sp_{2n}(p)$ and at there there is only one linear character). On the other hand $\det(\Phi_2|_{\Sigma})=\det(\lambda\Phi_1|_{\Sigma})=\lambda^{(q-1)/2}\det(\Phi_1|_{\Sigma})=\lambda^{(q-1)/2}$. Therefore we get that $6|(3^n-1)/2$ therefore $3|3^n-1$, a contradiction. So $\lambda\in\Irr(C_n)$ rational and nontrivial, so the unique character of order $2$. We have constructed $\chi^*\in\Irr(\Sigma)$ such that for$(\delta,\sigma)\in \Delta\times \mathcal{H}$, 

$$
(\chi^*)^{(\delta,\sigma)}=(\chi^*)^{\sigma}=\begin{cases} 
      \chi^* & \textrm{if   } \left(\dfrac{\ell}{p}\right)=1 \textrm{ or  ($2|f$  and}  \left(\dfrac{\ell}{p}\right)=-1 ). \\
      
      \lambda\chi^* & \textrm{if ($2 \nmid f$ and }  \left(\dfrac{\ell}{p}\right)=-1.)
   \end{cases}
$$

where $\lambda$ is the unique character of order 2 of $C_n$.\\

We now treat the $U$ part. Start by noticing that $\Delta=\mathbf{N}_{\Sigma}(U)=\langle U,\hat{\alpha} \rangle$ and $\Delta/U\cong C_n$. Note that $\Delta$ induces $\Gamma_{\chi^{\mathcal{H}}}=\Gamma_{\chi}$ when it acts on $U$ by conjugation. As above, the natural map $c:=\Delta \to \Gamma_{\chi}$ has kernel $C_2=\{1,z\}=\mathbf{Z}(X)= \mathbf{C}_{\Delta}(X)$.
We have that $\hat{\alpha}$ acts on $U$ the same way $\alpha\in\Aut(\SU_2(q))$ acted on $U\subseteq\SU_2(q)$. So, denoting by $b,u$ the images of $b,u\in \SU_2(q)$ we have that $U=\langle b,u\rangle$ satisfying the same relations, and $b^{\hat{\alpha}}=b^p, u^{\hat{\alpha}}=u$. The element $c(\hat{\alpha})$ which we have been denoting by $\hat{\alpha}$ has order $2n$ so $\hat{\alpha}$ has order divisible by $2n$ in $\Delta$. Since $\hat{\alpha}^{2n}\in \Ker(c)$ we have that $o(\hat{\alpha})=2n$ or $o(\hat{\alpha})=4n$ as an element of $\Delta$. If the order is $2n$, then, since $\hat{\alpha}^n\in U$ it has to be that $\hat{\alpha}^n=z$ but $b^z=b$ and $b^{\hat{\alpha}^n}=b^{-1}$ so $\hat{\alpha}$ has order $4n$ in $\Delta$. Let $\Delta_0/U$ de the unique $C_s$ subgroup of $\Delta/U$, note that $\Delta_0=\langle \hat{\alpha}^m,U\rangle$ and $o(b^2)=(q+1)/2$ which is not divisible by 2, and $o(\hat{\alpha}^m)=4s$ so it follows that we have an internal decomposition $\Delta_0=\langle b^2 \rangle\rtimes \langle\hat{\alpha}^m\rangle$ (note that $|\Delta:\Delta_0|=m$, $|\Delta_0|=2(q+1)s$). Denote by $\mu$ any (but fixed) of the two characters $\mu^+,\mu^-$ of $U$. Our character $\mu$ extends to $s$ characters of $\Delta_0$ and we want to have control of two of them. Note that $|\Lin(\Delta_0)|=|\Delta_0:\Delta_0'|\leq |\Delta_0:U'|=|\Delta_0:U||U:U'|=4s$ so we understand the group of linear characters of $\Delta_0$ as the linear characters extending $1_{\langle b^2\rangle}$ and simultaneously as the $4s$ linear characters extending the 4 linear characters of $U$ (we see that $\Lin(\Delta_0)\cong C_{4s}$). Note that any extension of $\mu$ has order $4s$. Fix $\mu^{\sharp}$ one such extension so that it also extends to $\Delta$ (do it by extending to $\Delta$ first and then restricting). Note that $\mu^{\sharp}(a)=\xi_{4s}$ for some $a\in\Delta_0$ where $o(\xi_{4s})=4s$ as a complex root of unity. \\ \\
\underline{Claim:} $\left(\dfrac{\ell}{p}\right)=1$ if and only if $\ell\equiv 1\mod 4s$ and $\left(\dfrac{\ell}{p}\right)=-1$ if and only if $\ell\equiv 2s+1\mod 4s$. 

To see this we start by noticing that the fact that $\ell|(q+1)$ implies $\ell\equiv 1 \mod 4s$ or $\ell \equiv 2s+1 \mod 4s$. Now assume that $\ell\equiv 1 \mod 4s$, then there exists an element $y\in\mathbb{F}_{\ell}$ of order $4s$. Let $t=p^m$. It is enough to show $\left(\dfrac{t}{\ell}\right)=1$ (note that by quadratic reciprocity  $\left(\dfrac{p}{\ell}\right)=1$ if and only if $\left(\dfrac{\ell}{p}\right)=1$). Note that $t^{2^k}\equiv-1 \mod \ell$. From this, we see that $t$ is a square modulo $\ell$. \\Conversely, if $\left(\dfrac{\ell}{p}\right)=1 $, since $(p^m)^s\equiv -1 \mod \ell$ and $p\equiv x^2\mod\ell$, we have that $-1\equiv y^{2s}\mod \ell$ for some $y$. From this, since $s$ is power of 2 we have that there is an element of order $4s$ in $\mathbb{F}_{\ell}$, so $4s|(\ell-1)$ and the claim is proved.\\

Thus, thinking about the extensions of $\mu$ in terms of where $\hat{\alpha}^m$ is sent, fix $\sigma\in\mathcal
{H}$ and $f$ as above in the proof, using the previous claim we have that since $\mu^{\sharp}(\hat{\alpha}^m)=\xi_{4s}$, if \begin{itemize}
    \item[(A)]  $\left(\dfrac{\ell}{p}\right)=1$ or $\left(\dfrac{\ell}{p}\right)=-1$ and $f$ is even, then $(\mu^{\sharp})^{\sigma}=\mu^{\sharp}$.
    \item[(B)] $\left(\dfrac{\ell}{p}\right)=-1$ and $f$ is odd, then $(\mu^{\sharp})^{\sigma}=\mu^{\sharp'}$ where $=\mu^{\sharp'}(\hat{\alpha}^m)=-\xi_{4s}$
\end{itemize}
 Let, by Theorem A \cite{Nav} since $m$ is odd, $\mu^*\in \Irr(\Delta)$ be the unique $\mathbb{Q}_{4s}$-valued extension of $\mu^{\sharp}$ to $\Delta$.  Let $(\delta,\sigma)\in (\Delta\times \mathcal{H})= (\Delta\times \mathcal{H})_{\mu}$, then $(\mu^*)^{(\delta,\sigma)}=(\mu^*)^{\sigma}$. If we are in situation (A), then $(\mu^*)^{\sigma}$ extends $\mu^{\sharp}$ and it is $\mathbb{Q}_{4s}$-valued so by uniqueness $(\mu^*)^{\sigma}=\mu^*$. If we are in situation (B) then $(\mu^*)^{\sigma}$ is the unique $\mathbb{Q}_{4s}$-valued extension of $\mu^{\sharp'}$ to $\Delta$, which we denote by $\mu^{*'}$. We conclude by understanding which $\lambda\in \Irr(\langle\hat{\alpha}U\rangle)=\Irr(C_n)$ satisfies $\mu^*=\lambda\mu^{*'}$. Note that $\lambda|_{\langle\hat{\alpha}^mU\rangle}$ is rational since $\lambda(\hat{\alpha}^m)=-1$, and $\lambda$ is certainly $\mathbb{Q}_{4s}$-valued. By Theorem A \cite{Nav}, it follows that $\lambda$ has to be rational valued and since nontrivial it is the unique character of $C_n$ of order two. Hence, we have constructed $\mu^*\in\Irr(\Delta)$ such that for$(\delta,\sigma)\in \Delta\times \mathcal{H}$, 

$$
(\mu^*)^{(\delta,\sigma)}=(\mu^*)^{\sigma}=\begin{cases} 
      \mu^* & \textrm{if   } \left(\dfrac{l}{p}\right)=1 \textrm{ or  ($2|f$  and}  \left(\dfrac{l}{p}\right)=-1 ). \\
      
      \lambda\mu^* & \textrm{if ($2 \nmid f$ and }  \left(\dfrac{l}{p}\right)=-1.)
   \end{cases}
$$
where $\lambda$ is the unique character of order 2 of $C_n$. Thus,  since both characters $\chi^*$, $\mu^*$ lie over the same character of $\mathbf{Z}(X)$,  we can finally conclude that $(X\rtimes \Gamma_{\eta^{\mathcal{H}}}, X, \eta)_{\mathcal{H}} \succeq_{\mathbf{c}}(U\rtimes\Gamma_{{\eta}^{\mathcal{H}}}, U, \mu)_{\mathcal{H}}$. \\

This concludes the proof of Claim 2 so the theorem follows.
\end{proof}

The following theorem treats $\PSL_2(q)$ when $q$ is even. The proof is a much simpler version of the proof of Theorem \ref{GM odd}. The character table of $\SL_2(q)$ when $q$ is even is loosely speaking formed by removing some rows and columns from the odd case. In particular, we do not have characters of degree $(q\pm1)/2$, we work with dihedral groups on the local side and also, a big part of the work to prove the central orderings can be recycled from Theorem 8.1 in \cite{BKNT}.
\begin{teo}\label{Sl2 2}
Let $p=2$ and let $q=p^n$. Then the simple groups $S=\SL_2(q)$ satisfy the inductive Galois-McKay condition for any prime $\ell$ dividing $|S|$.
 \end{teo}
 \begin{proof}
 First, we note that the case $\SL_2(4)\cong \PSL_2(5)$ was already treated so we will assume that $n\geq 3$. Also, we have already mentioned that for $\ell=2$ the proof is already in the literature, so we will assume that $\ell|(q-1)$ and that $\ell|(q+1)$.
 The universal covering group of $S$ is $X=S=\SL_2(q)$, $\Aut(X)=\langle X, \alpha \rangle=X\rtimes \langle\alpha \rangle$ where $\alpha$ acts by taking power $ 2$ to the matrix entries. We will use notation from Theorem 38.2 of \cite{D} for the irreducible characters and conjugacy classes of $X$. \\

 Assume first that $\ell|(q-1)$. Consider the diagonal torus (split case) $T=<a>\cong C_{q-1}$, we take $L\in \Syl_{\ell}(X)\subseteq T$. Let $U=\mathbf{N}_X(T)=\langle a,u |$ $a^{q-1}=1,$ $u^2=1$, $a^u=a^{-1}\rangle$, note that $U$ is dihedral of order $2(q-1)$, and $u$ is in the unique class of involutions of $X$. In this case, since $\mathbf{C}_X(L)=T$ (see the proof of 38.1 \cite{D}), we we can argue as in the proof of Theorem 1.4.3 of \cite{Bon} to show that  $U=\mathbf{N}_X(L)$. Hence, $U$ is self-normalizing and intravariant in $X$. 
 Since $U$ is $\langle  \alpha \rangle $-invariant we have that $\Aut(X)_U=\langle U, \alpha \rangle $. Now, since $L$ is also  $\langle \alpha \rangle $-invariant we have that  $\Gamma := \Aut(X)_L=\langle U, \alpha \rangle = \Aut(X)_U$. \\

 Note that $\Irr_{\ell'}(X)=\{1_X,\Psi,\chi_i:$ $1\leq i\leq (q-2)/2\}$. Also, note that $\Irr_{\ell'}(U)=\Irr(U)= \{ 1_U,\Sgn,\lambda_t\thinspace | \thinspace 1\leq t\leq (q-2)/2\}$, where $\lambda_t: 1\to 2,$ $a^l\to \rho^{tl}+\rho^{-tl}$ where $o(\rho)=q-1$, $c\to 0$ for $1\leq l\leq q-2$ and $(1_A)^U=1_U+\Sgn$. We give the next bijection viewing the index $i$ of $\lambda_i$ as in $\mathbb{Z}/(q-1)\mathbb{Z}$. 

 \begin{center}
     $\Omega:$ $\Irr_{\ell'}(X)\longrightarrow \Irr_{\ell'}(U)$, $
     1_X \longmapsto 1_U$, $\Psi \longmapsto \Sgn$,  $\chi_i\longmapsto\lambda_{2i}.$
 \end{center}
 \underline{Claim 1:} $\Omega$ is $\Gamma \times \mathcal{H}$-equivariant.
 
 We note that $1_X, \Psi$ and $1_U, \Sgn_U$ are $\Gamma\times \mathcal{H}$-stable. Indeed, $\Psi$ is the unique irreducible character of $X$ of degree $q$ and  $\Sgn$ is the unique nontrivial linear character of $U$. We need to show that the restriction of $\Omega$ to $\{\chi_i|$ $1\leq i\leq (q-2)/2\}$ is $\Gamma\times\mathcal{H}$-equivariant. For this we first show $\chi_i\to\lambda_i$ is equivariant (with $1\leq i\leq (q-2)/2$) and then show (with indices viewed
 in $\mathbb{Z}/(q-1)\mathbb{Z}$) that $\lambda_i\to \lambda_{2i}$ is equivariant. Let $(\delta,\sigma)\in \Gamma\times \mathcal{G}$ where $\delta\in \alpha^{-s}
 U$ with $s\in \mathbb{Z}$. Now since $
 U$ fixes both $\chi_j, \lambda_j$, also $\chi_j(a)= \lambda_j(a)$, and the subindex $j$ of $\chi_j$ (resp. $ \lambda_j$) is determined by the value at $a$. Note that $\chi_j(a^{2^m})= \lambda_j(a^{2^m})$ for $m\geq 0$ since $q-1\nmid 2^{m}$. Thus, we have that $\chi_j^{(\delta,\sigma)}(a)=\sigma(\chi_j(a^{2^s}))=\sigma( \lambda_j(a^{2^s}))= \lambda_j^{(\delta,\sigma)}(a)$, it follows that $\chi_i\to\lambda_i$ is $\Gamma\times\mathcal{G}$-equivariant. Now, fix $k$ coprime to $q-1$ such that $\sigma(\rho)=\rho^k$, note that $\lambda_i^{(\delta,\sigma)}=\lambda_{2^sik}$, it follows that $\lambda_{2i}^{(\delta,\sigma)}=\lambda_{2^{s+1}ik}$ proving Claim 1. To verify the inductive Galois-McKay condition for $\ell$ it remains to prove the following.\\
 \underline{Claim 2:} For every $\chi\in\Irr_{\ell'}(X)$ we have that
 $$
 (X\rtimes \Gamma_{\chi^{\mathcal{H}}}, X, \chi)_{\mathcal{H}} \succeq_{\mathbf{c}}(U\rtimes\Gamma_{\chi^{\mathcal{H}}}, U, \Omega(\chi))_{\mathcal{H}}.
 $$
 
 We will use the letter $\mu$ to denote $\Omega(\chi)$ for a fixed $\chi$. Observe that for every $\chi\in\Irr_{\ell'}(X)$ by Claim 1 we have that $(\Gamma \times \mathcal{H})_{\chi}=(\Gamma \times \mathcal{H})_{\mu}$. Let us first consider  $\chi\in \{1_X,\Psi\}$, one can verify directly that $[ \chi_U, \Omega(\chi)]=1$ so our conclusion holds by Theorem \ref{crit}(a).\\
 Now consider $\chi=\chi_j$ for some $1\leq j\leq (q-2)/2$. Assume first that $\chi_j\neq\chi_{(q-1)/3}$. By the proof of Theorem 8.1 \cite{BKNT}, $[\chi_U,\mu]=1$ so we are done by Theorem \ref{crit}(a). If $3|(q-1)$ and $\chi_j=\chi_{(q-1)/3}$, then we have $\mu=\mu_{2(q-1)/3}=\mu_{(q-1)/3}$ which comes from inducing $\xi\in\Irr(\langle a\rangle)$ of order 3. Both characters are rational and $\alpha$-invariant. Consider $\Sigma=\Aut(X)=X\rtimes\langle\alpha\rangle$ and let $\Delta=\mathbf{N}_{\Sigma}(U)=U\rtimes\langle\alpha\rangle$ which induces $\Gamma=\Gamma_{\chi}=\Gamma_{\chi^{\mathcal{H}}}$ when acts on $X$ via conjugation. If $n$ is odd consider $\mu^*\in\Irr(\Delta)$ the unique rational extension of $\mu$ to $\Delta$. If $n$ is even we consider an extension as follows. Note that $\Delta_{\xi}=\langle a\rangle \rtimes\langle u\alpha\rangle$. Note that $|\Delta:\Delta_{\xi}|=2$, if we extend $\xi$ to $\xi^+ \in \Irr(\Delta_{\xi})$ by letting $\xi^+(u\alpha)=1$ and then we consider the induced character $\mu^*=(\xi^+)^{\Delta}$ we have that this character extends $\lambda$ by Mackey's formula. Moreover it is rational, to see this note that $[\alpha,u]=1$, take $\{1,u\}$ as a transversal and apply the formula given right before Lemma 5.2 of \cite{Is}. For $\chi$, apply (since the degree is odd and the determinant is 1) the proof of Corollary 2.4 of \cite{NT} to find a rational valued character $\chi^*\in \Irr(\Aut(X))$ extending $\chi$. We can now conclude by Theorem \ref{crit}(c) since $\mathbf{C}_{\Aut(X)}(X)=1$. This concludes the proof of Claim 2, hence the proof for $\ell$ dividing $q-1$.\\
 
 Assume now that $\ell$ divides $q+1$. We will work with the non-split torus $C_{q+1}$ and for that reason it is convenient to see $X$ as $\SU_2(q)$. Formally, let $V=\langle e,f\rangle_{\mathbb{F}_{q^2}}$, with the Hermitian product $e\circ e=1=f\circ f$, and $e\circ f=0$.  We write $\Aut(X)=
 \langle X, \alpha \rangle$ where  $\alpha$ acts by taking power $2$ to the matrix entries after fixing an orthonormal basis. For the character table we will use notation from Theorem 38.1 of \cite{D}. Let $\xi\in \mathbb{F}_{q^2}$ have order $q+1$
 \begin{equation*} b = \begin{pmatrix}\xi & 0\\ 0 & \xi^{-1}\end{pmatrix}, u = \begin{pmatrix}0 & 1\\ 1 & 0\end{pmatrix}\end{equation*}

 Consider the torus (non-split case) $T=<b>\cong C_{q+1}$, we take $L\in \Syl_{\ell}(X)\subseteq T$. Let $U=\mathbf{N}_X(T)=\langle b,u |$ $b^{q+1}=1,$ $u^2=1$, $b^u=b^{-1}\rangle$, note that $U$ is dihedral of order $2(q+1)$, arguing as we did in the first part of the proof $U=\mathbf{N}_X(L)$ so it is self-normalizing and intravariant in $X$. Since $U$ is $\langle \alpha \rangle $-invariant we have that $\Aut(X)_U=\langle U, \alpha \rangle $. Now, since $L$ is also  $\langle \alpha \rangle $-invariant we have that  $\Gamma := \Aut(X)_L=\langle U, \alpha \rangle = \Aut(X)_U$. 

 Note that $\Irr_{\ell'}(X)=\{1_X,\Psi,\theta_i:$ $1\leq i\leq q/2\}$. Also, note that $\Irr_{\ell'}(U)=\Irr(U)= \{ 1_U,\Sgn,\lambda_t\thinspace | \thinspace 1\leq t\leq q/2\}$, where $\lambda_t: 1\to 2,$ $b^l\to \nu^{tl}+\nu^{-tl}$ where $o(\nu)=q+1$, $c\to 0$ for $1\leq l\leq q$ and $(1_B)^U=1_U+\Sgn$. We give the next bijection viewing the index $i$ of $\lambda_i$ as in $\mathbb{Z}/(q+1)\mathbb{Z}$. 

 \begin{center}
     $\Omega:$ $\Irr_{\ell'}(X)\longrightarrow \Irr_{\ell'}(U)$, $
     1_X \longmapsto 1_U$, $\Psi \longmapsto \Sgn$, $\theta_i\longmapsto\lambda_{2i}.$
 \end{center}
 \underline{Claim 1:} $\Omega$ is $\Gamma \times \mathcal{H}$-equivariant.
 
 As in the first part, $1_X, \Psi$ and $1_U, \Sgn_U$ are $\Gamma\times \mathcal{H}$-stable. To complete the proof of Claim 1 we just need to show that the restriction of $\Omega$ to $\{\theta_i|$ $1\leq i\leq q/2\}$ is $\Gamma\times\mathcal{H}$-equivariant. The proof of this goes exactly as the previous case. We have proved claim 1. To verify the inductive Galois-McKay condition for $\ell$ it remains to prove the following.\\ \underline{Claim 2:} For every $\chi\in\Irr_{\ell'}(X)$ we have that
 $$
 (X\rtimes \Gamma_{\chi^{\mathcal{H}}}, X, \chi)_{\mathcal{H}} \succeq_{\mathbf{c}}(U\rtimes\Gamma_{\chi^{\mathcal{H}}}, U, \Omega(\chi))_{\mathcal{H}}.
 $$
 
 We will use the letter $\mu$ to denote $\Omega(\chi)$ for a fixed $\chi$. Observe that for every $\chi\in\Irr_{\ell'}(X)$ by Claim 1 we have that $(\Gamma \times \mathcal{H})_{\chi}=(\Gamma \times \mathcal{H})_{\mu}$. Assume that $\chi=1_X$, then  $[\chi,\Omega(\chi)]=1$, hence we are done by Theorem \ref{crit}(a). Now take $\chi=\Psi$, it is easy to compute that $[\chi,\Omega(\chi)]=0$, hence we do the following. Consider  the group $\Sigma=\Aut(X)\supseteq X$. Then $\Delta=\mathbf{N}_{\Sigma}(U) $ induces $\Gamma=\Gamma_{\chi}=\Gamma_{\chi^{\mathcal{H}}}$. Again, we use the fact (see \cite{MS}) that the Steinberg character has a rational valued extension $\chi^*$ to $\Sigma$, also, since $\mu$ is linear and of order $2$, if we write $n=sm$ with $s=2^k$ and $m$ odd, we can find making use of Theorem A \cite{Nav} a $\mathbb{Q}_{2s}$ valued extension $\mu^*$ of $\mu$ to $\Delta$. Of course, $\chi^*$ is $\Delta\times \mathcal{H}$-invariant and so is $\mu^*$ (since $\ell|q+1$, $\mathcal{H}$ fixes $\mathbb{Q}_{2s}$ point-wise), hence we are done by Theorem \ref{crit} (c) since $\mathbf{C}_{\Sigma}(X)=1$.\\
 Now, let $\chi=\theta_j$ for some $1\leq j\leq q/2$, assume first that $\theta_j\neq\theta_{(q+1)/3}$. By the proof of  Theorem 8.1 \cite{BKNT}, we have that $[\chi,\mu]=1$ so by Theorem \ref{crit}(a), 
 $(X\rtimes \Gamma_{\theta_j^{\mathcal{H}}}, X, \theta_j)_{\mathcal{H}} \succeq_{\mathbf{c}}(U\rtimes\Gamma_{{\theta_j}^{\mathcal{H}}}, U, \mu_j)_{\mathcal{H}}$. Now assume that  $3|(q+1)$ and $\theta_j=\theta_{(q+1)/3}$, then we have that $\mu=\mu_{2(q+1)/3}=\mu_{(q+1)/3}$ which comes from inducing $\xi\in\Irr(\langle b\rangle)$ of order 3. Both characters are rational and $\alpha$-invariant. 
 Assume first that $n$ is odd, consider $\Sigma=X\rtimes \langle \alpha^2\rangle$ which induces $\Aut(X)$ and $\mathbf{C}_{\Sigma}(X)=1$. Let $\Delta=\mathbf{N}_{\Sigma}(U)=U\rtimes\langle \alpha^2\rangle$, which induces $\Gamma=\Gamma_{\chi}=\Gamma_{\chi^\mathcal{H}}$ when acting on $X$ via conjugation. It is enough to consider  the unique rational extension of $\mu,\chi$ to $\Delta,\Sigma$ respectively and apply Theorem \ref{crit}(c). If $n$ is even, note that $3\nmid q+1$ so we are done.
 \end{proof}

Finally, we check that the inductive Galois-McKay condition holds for $S=\PSL_2(9)=A_6$ for the prime 5. This is the only prime left to check on the literature for $A_6$, with this, we will conclude the proof of Theorem C. Full details of the next proof are available on request.
\begin{teo}
The group $A_6$ verifies the inductive Galois-McKay condition for all primes $\ell$ dividing $|A_6|$.
\end{teo}
\begin{proof}
    
Note that $|A_6|=360=5\cdot2^3\cdot3^2$, so we may assume $\ell=5$. The universal cover of $S$ is $6.A_6$. The group $X$ is provided by GAP (\cite{GAP}), we have that $\Aut(X)\cong \Aut(S)$, moreover, table 6.3.1 of \cite{GLS} reveals that the action of $\Out(X)=C_2\times C_2$ on $\mathbf{Z}(X)$ is nontrivial. We have that $\Aut(X)=\langle A_6,\tau,\alpha\rangle$ where (the images under canonical projection of) $\tau,\alpha$ generate $\Out(X)$. Let $P\in \Syl_5(X)$ and let $U=\mathbf{N}_X(P)$ which is self-normalizing and intravariant on $X$.  It is easy to see that  (multiplying the outer automorphisms by an inner automorphism if needed) $\tau,\alpha$ fix $U$, thus, $\Gamma=\Aut(X)_U=\langle \bar{U}, \tau, \alpha \rangle$ where $\bar{U}$ denotes reduction modulo the center (this group is in the normalizer of a Sylow$-5$ subgroup of $A_6$ and the same holds for reductions modulo subgroups of the center). Also, we may assume that the group $6.A_6.2_1$ induces $\tau$ and the group $6.A_6.2_2$ induces $\alpha$, these groups are also in \cite{GAP} and we have used them to apply Theorem \ref{crit}. Using computations with GAP we can explicitly give a $\Gamma\times\mathcal{H}$-equivariant bijection between $\Irr_{5'}(X)$ and $\Irr_{5'}(U)$ satisfying $(X\rtimes \Gamma_{\chi^{\mathcal{H}}}, X, \chi)_{\mathcal{H}} \succeq_{\mathbf{c}}(U\rtimes\Gamma_{\chi^{\mathcal{H}}}, U, \Omega(\chi))_{\mathcal{H}}.$ One can pair faithful characters of $X$ of degree 12 with faithful characters of $U$ of degree 2, and use Theorem \ref{crit}(e). The faithful characters of degree 6 are paired with the unique linear characters of $U$ such that their kernel only contains elements of order 1 and 5 in $U$, then one uses Theorem \ref{crit}(a). The characters of $X$ with kernel $C_2$ (of degrees 3,6,9) are seen as characters of $3.A_6$, they are paired with suitable characters of $U$ and we have checked the central order condition with applications of Theorem \ref{crit}(a), \ref{crit}(c) (to apply \ref{crit}(c), for the characters of order 6 we had to do some computations with $\Sigma=3.A_6.2^2$, also provided by \cite{GAP}). The remaining characters can be seen as characters of $\SL_2(9)$ and we can recycle the bijection from Theorem \ref{GM odd}.  
\end{proof}

\end{document}